\pdfoutput=1
\documentclass[12pt,reqno]{amsart}
\usepackage[T1]{fontenc}
\usepackage{microtype} 
\usepackage[margin=1in]{geometry}

\usepackage{pkgfile}

\makeatletter
\newtheorem*{rep@theorem}{\rep@title}
\newcommand{\newreptheorem}[2]{%
\newenvironment{rep#1}[1]{%
 \def\rep@title{#2 \ref{##1}}%
 \begin{rep@theorem}}%
 {\end{rep@theorem}}}
\makeatother

\newtheorem{theorem}{Theorem}[section]
\newtheorem{proposition}[theorem]{Proposition}
\newtheorem{lemma}[theorem]{Lemma}
\newtheorem{corollary}[theorem]{Corollary}
\newtheorem{conjecture}[theorem]{Conjecture}

\newreptheorem{theorem}{Theorem}

\theoremstyle{definition}

\theoremstyle{remark}
\newtheorem*{remark}{Remark}

\newcommand{\dd}{\delta}

\newcommand{\CC}{\mathbb{C}}
\newcommand{\EE}{\mathbb{E}}

\newcommand{\PP}{\mathbb{P}}
\newcommand{\RR}{\mathbb{R}}

\newcommand{\ZZ}{\mathbb{Z}}

\DeclareMathOperator{\GW}{GW}
\DeclareMathOperator{\DiscDerv}{DiscDerv}
\DeclareMathOperator{\DiscGrad}{DiscGrad}
\DeclareMathOperator{\DLip}{DLip}
\DeclareMathOperator{\DiscGW}{DiscGW}
\DeclareMathOperator{\Bernoulli}{Bernoulli}
\DeclareMathOperator{\Normal}{Normal}

\begin{document}

\title[Upper tail large deviations for arithmetic progressions]{Upper tail large deviations for arithmetic progressions in a random set}

\author[Bhattacharya]{Bhaswar B. Bhattacharya}\address{Department of Statistics\\ University of Pennsylvania\\ Philadelphia, PA 19104 \\ USA}
\email{bhaswar@wharton.upenn.edu}

\author[Ganguly]{Shirshendu Ganguly}
\address{Department of Statistics\\ UC Berkeley \\ California, CA 94720 \\ USA}
\email{sganguly@berkeley.edu}
\thanks{SG was supported by a Miller Research Fellowship}

\author[Shao]{Xuancheng Shao}
\address{Department of Mathematics \\ University of Kentucky 
\\ Lexington, KY 40506 \\ USA}
\email{xuancheng.shao@uky.edu}
\thanks{XS was supported by a Glasstone Research Fellowship.}

\author[Zhao]{Yufei Zhao}
\address{Department of Mathematics \\ Massachusetts Institute of Technology\\ Cambridge, MA 02139 \\ USA}
\email{yufeiz@mit.edu}
\thanks{YZ was supported by an Esm\'ee Fairbairn Junior Research Fellowship at New College, Oxford and NSF Award DMS-1362326.}

\begin{abstract}
Let $X_k$ denote the number of $k$-term arithmetic progressions in a random subset of $\mathbb{Z}/N\mathbb{Z}$ or $\{1, \dots, N\}$ where every element is included independently with probability $p$. 
We determine the asymptotics of $\log \mathbb{P}(X_k \ge (1+\delta) \mathbb{E} X_k)$ (also known as the large deviation rate) where $p \to 0$ with $p \ge N^{-c_k}$ for some constant $c_k > 0$, which answers a question of Chatterjee and Dembo. 
The proofs rely on the recent nonlinear large deviation principle of Eldan, which improved on earlier results of Chatterjee and Dembo. 
Our results complement those of Warnke, who used completely different methods to estimate, for the full range of $p$, the large deviation rate up to a constant factor.
\end{abstract}

\maketitle

\setcounter{tocdepth}{1}
\tableofcontents

\section{Introduction}

Let $X_k$ denote the number of $k$-term arithmetic progressions ($k$-AP) in the random set $\Omega_p$, where $\Omega$ is taken to be either $\Z/N\Z$ or $[N] := \{1, \dots, N\}$ throughout this paper, and $\Omega_p$ denotes the random subset of $\Omega$ where every element is included independently with probability $p$. The \emph{upper tail problem} asks to estimate the probability that $X_k$ significantly exceeds its expectation.  This problem has received some interest over the years~\cite{CD,JR11,War}. More generally, the problem of computing tail probabilities of a sum of weakly dependent random variables has a long and interesting history \cite{BGLZ,Cha12,CD,CD10,CV11,DK12b,DK12,JOR04,JR02,JR04,JR11,KV04,LZ-sparse,LZ15,Vu01,War}. There has been many exciting recent developments, particularly in the setting of random graphs~\cite{Cha12,CD,DK12}, where one is interested in the concentration of the number of triangles in an Erd\H{o}s--R\'enyi random graph $\cG(N, p)$. See Chatterjee's recent survey~\cite{ChaBAMS} and the references therein for an introduction to recent developments on large deviations in random graphs.

The recent work of Warnke~\cite{War} settles the question of the asymptotic order of $\PP(X_k \ge (1+\delta)\EE X_k)$ when $\Omega=[N]$. Warnke~\cite{War} shows that for fixed $\delta > 0$ and $k\ge 3$, there exists constants $c, C > 0$ (depending only on $k$) such that 
\begin{equation}\label{eq:logP-order}
p^{C \sqrt{\delta}  N p^{k/2}} \leq \PP(X_k \ge (1+\delta) \EE X_k) \leq  p^{c \sqrt{\delta}  N p^{k/2}},
\end{equation}
as long as $p = p_N \ge (\log N/N)^{1/(k-1)}$ and $p$ is bounded away from 1. 
Prior to Warnke's work, the best upper bound~\cite{JR11} was $\PP(X_k \ge (1+\delta) \EE X_k) \leq e^{-c_\delta p^{k/2} N}$, for some constant $c_\delta > 0$ depending on $k$ and $\delta$. However, the natural question of precise asymptotics still remained open. 
 
The main result of this paper shows for every $k \ge 3$, fixed $\delta>0,$ and $\Omega=[N],$ if $p=p_N \ge N^{-\frac{1}{6 k(k-1)}}\log N$ and $p \to 0$,  then, as $N \to \infty$,
\begin{equation}
	 \label{eq:logP-kAP-asymp}
 \PP(X_k \ge (1+\delta) \EE X_k)
=  p^{(1+o(1)) \sqrt{\delta} N p^{k/2}}.
\end{equation} 
The lower bound to the probability can be seen by forcing an interval of length $(1+o(1)) \sqrt{\delta}p^{k/2} N$  to be present in $\Omega$, so that it generates the extra $\delta \EE X_k$ many $k$-APs as desired. For the special case of $k=3$, which was also treated in \cite{CD}, methods in \cite{Eldan} combined with Fourier analysis allow us to take $p \ge N^{-1/18} \log N$, improving on the $p \ge N^{-1/162} (\log N)^{34/162}$ hypothesis in \cite{CD}. For $k\ge 4$, \eqref{eq:logP-kAP-asymp} is the first large deviation result for $k$-AP counts allowing $p$ to decay as $N^{-c}$,
thereby answering  a question posed by Chatterjee and Dembo \cite[Section~1.8]{CD} and improving on Warnke's result in the appropriate regime.\footnote{In a previous arXiv version of this paper, we proved the result with $p$ decaying extremely slowly, as our previous proof depended on the heavy-powered inverse theorem for Gowers uniformity norms due to Green, Tao, and Ziegler~\cite{GTZ12}.}

The proofs rely on the powerful nonlinear large deviation principle (LDP) developed by Chatterjee and Dembo~\cite{CD}, which was recently improved by Eldan~\cite{Eldan} using different methods, namely stochastic control theory. These LDPs reduce the determination of the large deviation rate (i.e., asymptotics of log-probability) in many combinatorial problems to a natural variational problem involving entropies.  For the problem of upper tails of subgraph counts in a sparse random graph, the corresponding variational problem was recently solved \cite{BGLZ,LZ-sparse}.
For arithmetic progressions, Chatterjee and Dembo were able to verify the hypotheses of their LDP for 3-term arithmetic progressions but not longer ones \cite[Section 1.5]{CD}. 
More recently, Eldan \cite{Eldan} provided a different, but related, set of hypotheses for his LDP, involving the supremum of an associated Gaussian process (see Theorem \ref{thm:eldan}). 
We prove the necessary bounds to apply Eldan's LDP for arithmetic progressions of arbitrary fixed length. We remark that similar arguments can also be used to verify Chatterjee and Dembo's LDP hypotheses for  arithmetic progressions of any fixed length. Some interesting open problems in additive combinatorics arise in the analysis of this Gaussian process (see Section~\ref{sec:gw}).

After establishing the LDP, we solve the corresponding variational problem. Here, there are two regimes, below, depending on how $\delta$ decays to zero compared to $p$.

\begin{itemize}

\item[(1)] In the \emph{macroscopic} (large $\delta$) regime, where $\delta^{-3} p^{k-2}(\log(1/p))^2 \to 0$, the solution of the variational problem reduces to the extremal problem of maximizing the number of $k$-APs in a set of given size, which was solved by Green and Sisask \cite{GS08} for 3-APs, and extended to $k$-APs in Theorem~\ref{thm:maxAP}. The solution to this extremal problem is attained by an interval. The case of fixed $\delta > 0$, namely the asymptotic \eqref{eq:logP-kAP-asymp}, belongs to this regime.

\item[(2)] In the \emph{microscopic} (small $\delta$) regime, where $\delta^{-3} p^{k-2}(\log(1/p))^2 \to \infty$, 
the variational problem exhibits a rather different qualitative behavior compared to the previous case. We show that
\[
\PP(X_k \ge (1+\delta) \EE X_k) = e^{-(c+o(1))\delta^2 N p},
\]
for some explicit constant $c > 0$ depending on $k$ and $\Omega$.
\end{itemize}

Note that in both these regimes above we require $p=p_N \rightarrow 0$. If both $p$ and $\delta$ are fixed, the situation is quite different, and we report some partial results for this setting in Section \ref{sec:replica}.  

\medskip

It is worth comparing these results for arithmetic progressions to the corresponding results for triangles in a random graph. Let $X_{K_3}$ denote the number of copies of $K_3$ in $\cG(N,p)$. It was shown in~\cite{LZ-sparse} (again relying on \cite{CD}, which was recently improved in \cite{Eldan}) that for $p = p_N \to 0$ with $p \ge N^{-1/18}\log N$ and fixed $\delta > 0$, one has
\[
\PP(X_{K_3} \ge (1+\delta) \EE X_{K_3}) = p^{(1+o(1)) \min\{\delta^{2/3}/2, \, \delta/3\} N^2p^2}.
\]
This result was extended in \cite{BGLZ}, which determined the upper tail large deviation rate of $X_H$  for every graph $H$. The extra complexity in the above expression, as compared to~\eqref{eq:logP-kAP-asymp}, arises from the dichotomy of methods of generating many extra triangles: we can either force a clique to be present, or force a small subset of vertices to be connected to all other vertices.

\section{Statements of results}
\label{sec:thm}

\subsection{Notation}
\label{sec:notation}
We recall some standard asymptotic notations. For two nonnegative sequences $(f_n)_{n\geq 1}$ and $(g_n)_{n\geq 1}$, $f_n \lesssim g_n$ means $f_n=O(g_n)$; $f_n \sim g_n$ means $f_n = (1+o(1))g_n$; and $f_n \asymp g_n$ means $f_n = \Theta(g_n)$, i.e., $f_n \lesssim g_n \lesssim f_n$. Subscripts in the above notation,  for example, $O_\square(\cdot)$, $\lesssim_\square$, $\asymp_\square$ denote that the hidden constants may depend on the subscripted parameters. We always treat $k$ (as in $k$-AP) as a constant, and the dependence of the hidden constants on $k$ is always implicitly assumed and may be suppressed in the asymptotic notation.

For any set $A$ in some ambient abelian group (in this paper the ambient group will always be either $\Z$ or $\Z/N\Z$) and $k \ge 3$, we write $T_k(A)$ to denote the number of pairs $(a,b)$ of elements in the ambient group such that $a,a+b,a+2b,\dots,a+(k-1)b \in A$. Note that every non-trivial $k$-AP is counted twice, and every trivial $k$-AP (i.e., $b=0$) is counted once. It will be convenient to state our results in terms of $T_k(A)$.

For $p\in (0, 1)$ and a subset $\Omega$ in the ambient group, denote by $\Omega_p\subset \Omega$ the random set obtained by independently including each element in $\Omega$ with probability $p$. Throughout the paper, we will consider two settings:
\begin{enumerate}
\item[(1)] $\Omega = [N] = \{1, \dots, N\}$, and
\item[(2)] $\Omega = \Z / N \Z$.
\end{enumerate}
The ambient group is $\Z$ and $\Z/N\Z$ in the two respective settings. Note that, as long as $p^{k-1} N \to \infty$, it is easy to see that $\EE T_k(\Omega_p) = p^k(T_k(\Omega) - |\Omega|) + p|\Omega| \sim p^k T_k(\Omega)$. In this paper we are interested in the {\it upper tail probability},  $\P(T_k(\Omega_p)\geq (1+\delta) \E T_k(\Omega_p)),$ when $p=p_N\rightarrow 0$, as $N\rightarrow \infty$ ($\delta = \delta_N$ may also depend on $N$).

The relative entropy function with respect to $\Bernoulli(p)$ is denoted
\[
I_p(x) := x \log \frac{x}{p} + (1-x) \log\frac{1-x}{1-p}.
\]
Finally, denote the weighted $k$-AP count of a function $f: \Omega\rightarrow \R$ by
\begin{align}
T_k(f) := \sum_{a,b} f(a)f(a+b)\dotsm f(a+(k-1)b).
\label{kap}
\end{align}
Here $a$ and $b$ each range over all elements of the ambient group (either $\Z$ or $\Z/N\Z$) such that $\{a,a+b,\dots,a+(k-1)b\}\subset \Omega$. By convention, when $\Omega = [N]$, we set $f(x) = 0$ for all $x \notin \Omega$.
Note that $T_k(A) = T_k(\bm 1_A)$, where $\bm 1_A$ is the indicator function of $A$.

\subsection{Large deviation principle}
\label{sec:statements-ldpkap}

Let us write
\begin{equation} \label{eq:var}
\phi_p^{(k,\Omega)}(\delta) :=  \mkern-18mu \inf_{f \colon \Omega \to [0,1]}\left\{ \sum_{a \in \Omega} I_p(f(a)):  T_k(f) \ge (1+\delta)p^k T_k(\Omega) \right\}
\end{equation}
for the natural large deviations variational problem for upper tails of $k$-AP counts. 
We will establish in Section~\ref{sec:ldpgw} the following LDP for $k$-APs, via Eldan's LDP~\cite{Eldan}.

\begin{theorem} \label{thm:ldpkap}
Fix $k \ge 3$. Let $\Omega = [N]$ or $\Z/ N \Z$. Let $p = p_N$ be bounded away from $1$, and $\delta = \delta_N > 0$ with $\delta = O(1)$ such that
\begin{equation} \label{eq:hyp-p-lower}
\min\{\delta p^k, \delta^2p\} \ge N^{-\frac{1}{6(k-1)}} \log N.
\end{equation}
Then, as $N \to \infty$,
\begin{equation}\label{eq:ldp-kap}
- \log \PP(T_k(\Omega_p) \ge (1 + \delta)\EE T_k(\Omega_p)) \\
= (1+o(1))\phi_p^{(k,\Omega)}(\delta + o(\delta)).
\end{equation}
Furthermore, for $k = 3$, the right-hand side of \eqref{eq:hyp-p-lower} can be relaxed to $N^{-1/6}(\log N)^{7/6}$; for $k=4$, it can be relaxed to $N^{-1/12} (\log N)^{13/12}$.
\end{theorem}

\begin{remark}
For fixed $\delta > 0$, the theorem requires $p$ to decay slower than $N^{-\frac{1}{6k(k-1)}}$. Very recently, Bri\"et and Gopi \cite{BG} improved the exponent from $\frac{1}{6k(k-1)}$ to $\frac{1}{6k\lceil (k-1)/2\rceil}$, which improves our result for all $k \ge 5$. It remains an open problem to extend the range of validity of $p$. In comparison, Warnke's asymptotics~(\ref{eq:logP-order}) on the order of log-probability holds for all $p \ge (\log N/N)^{1/(k-1)}$.	
\end{remark}

In the above theorem, $\delta$ is allowed to decay as a function of $N$, and there is a qualitative change in the behavior of $\phi_p^{(k,\Omega)}(\delta)$ depending on how quickly $\delta$ decays compared to $p$. Drawing a parallel from statistical physics\footnote{The terms macroscopic/microscopic appear in many contexts in the statistical mechanics literature, and is generically used to described large/small scale behaviors, respectively. However, the exact definitions vary depending on the problem in question.}, we refer to the two regimes by:
\begin{itemize}
\item \emph{Macroscopic scale}: when $\delta$ is ``large'', namely when $\delta^{-3} p^{k-2}(\log(1/p))^2 \to 0$; and
\item \emph{Microscopic scale}: when $\delta$ is ``small'', namely when $\delta^{-3} p^{k-2}(\log(1/p))^2 \to \infty$.
\end{itemize}

\subsection{Macroscopic scale}
\label{sec:statements-macro}

In the macroscopic scale, $\delta^{-3} p^{k-2}(\log(1/p))^2\rightarrow 0$. Here, it might be helpful to think of $\delta$ as a constant or tending to zero ``slowly'' compared to $p$. We will establish, in Section~\ref{sec:macro}, the following asymptotic solution to the variational problem \eqref{eq:var} in this regime.

\begin{theorem} \label{thm:rate}
Fix $k\ge 3$. Let $\Omega = [N]$ or $\Z / N\Z$, and in the latter case assume that $N$ is prime. Let $p = p_N \to 0$ and $\delta = \delta_N > 0$ be such that $\delta = O(1)$, $\delta p^k N^2 \to \infty$, and $\delta^{-3} p^{k-2}(\log(1/p))^2 \to 0$. Then, as $N \to \infty$,
\[
\begin{split}
\phi_p^{(k,\Omega)}(\delta)
&=  (1+o(1)) \sqrt{(k-1)\delta p^k T_k(\Omega) }  \log(1/p)
\\
&= \begin{cases}
  (1+o(1))  \sqrt{\delta} p^{k/2} N \log(1/p)
  & \text{if } \Omega = [N], \\
  (1+o(1))  \sqrt{(k-1)\delta} p^{k/2} N \log(1/p)
  & \text{if } \Omega = \Z/N\Z \text{ with prime $N$}.
 \end{cases}
 \end{split}
\]
\end{theorem}

\begin{remark}
In the case of constant $\delta$, say, while the above solution to the variation problem needs only $p^k N^2 \to \infty$, a stronger condition $p^{k-1} N \to \infty$ (implying $\EE T_k(\Omega_p) \sim p^k T_k(\Omega)$) is necessary even just for the concentration of the random variable $T_k(\Omega_p)$.
\end{remark}

We prove Theorem~\ref{thm:rate} by first reducing the variational problem to an extremal problem in additive combinatorics, namely that of determining the size of the smallest subset of $\Omega$ with a given number of $k$-APs, or equivalently, the maximum number of $k$-APs in a subset of $\Omega$ of a given size.

\begin{proposition} \label{ppn:rate-extremal}
Under the same hypotheses as Theorem~\ref{thm:rate} (except that $N$ is not required to be prime in the case $\Omega = \Z/N\Z$), as $N \to \infty$,
\begin{align}\label{eq:rateext}
\phi_p^{(k,\Omega)}(\delta)
=  (1+o(1)) \log(1/p) \cdot \min_{S \subset \Omega} \Big\{ |S|: T_k(S) \ge \delta p^k T_k(\Omega)\Big\}.
\end{align}
\end{proposition}

The number of $k$-APs in a set of given (sufficiently small) size is always maximized by an interval, as stated precisely below. The theorem below was proved for 3-APs by Green and Sisask~\cite{GS08} and extended to all $k$-APs in Section~\ref{sec:kAPmax}.

\begin{theorem}\label{thm:maxAP}
Fix a positive integer $k \geq 3$. There exists some constant $c_k > 0$ such that the following statement holds. Let $A \subset \Z$ be a subset with $|A| = n$, or $A \subset \Z/N\Z$ with $N$ prime and $|A| = n \leq c_kN$. Then $T_k(A) \leq T_k([n])$.
\end{theorem}

After some algebra one easily obtains
\[ T_k([n]) = \frac{n^2}{k-1} - \frac{r^2}{k-1} + r, \]
where $r \in \{1,2,\dots,k-1\}$ is chosen such that $n \equiv r\pmod{k-1}$. In particular we have
\begin{equation}\label{eq:Tk(n)}
\frac{n^2}{k-1} \leq T_k([n]) \leq \frac{n^2}{k-1} + \frac{1}{4}(k-1).
\end{equation}
From this formula we can easily deduce Theorem~\ref{thm:rate} from Proposition~\ref{ppn:rate-extremal}.

Combining with the large deviation principle Theorem~\ref{thm:ldpkap}, we obtain the following corollary on the large deviation rate for upper tails of $k$-AP counts.

\begin{corollary}\label{cor:upper-tail-rate}
Fix $k \ge 3$. Let $\Omega = [N]$ or $\Z/ N \Z$, and in the latter case assume that $N$ is prime. 
Let $p = p_N \to 0$ and $\delta = \delta_N > 0$ be such that $\delta = O(1)$, $\delta^{-3} p^{k-2}(\log(1/p))^2 \to 0$, and
\begin{align}\label{eq:hyp-p-lowerII}
\delta p^k \ge N^{-\frac{1}{6(k-1)}} \log N
\end{align}
(for $k=3$, the right-hand side can be relaxed to $N^{-1/6}(\log N)^{7/6}$; for $k=4$, it can be relaxed to $N^{-1/12} (\log N)^{13/12}$). Then, as $N \to \infty$, the random variable $X_k = T_k(\Omega_p)$ satisfies
\[
\frac{- \log \PP\left(X_k \ge (1+\delta)\EE X_k\right)}{\sqrt{\delta} p^{k/2} N \log (1/p)}
= \begin{cases}
  1+o(1)
  & \text{if } \Omega = [N], \\
  \sqrt{k-1} + o(1)
  & \text{if } \Omega = \Z/N\Z \text{ with prime $N$}.
 \end{cases}
\]
\end{corollary}

\begin{remark} When $\Omega=\Z/N\Z$, it is standard to assume $N \to \infty$ along the primes to avoid torsion issues. Without the primality assumption, the leading constant in the large deviations rate function could depend on the subsequence along which $N$ goes to infinity (a similar issue was discussed in \cite{Croot08}). For example, when $N= N_1N_2$, the maximum number of $k$-APs in $\Z/N\Z$ in a set of size $N_1$ is given by $A = N_2\Z/N\Z$, which has $T_k(A) = N_1^2$, more than $T_k([N_1]) = N_1^2/(k-1) + O(1)$. Thus, as a consequence of Proposition~\ref{ppn:rate-extremal}, we could have a different constant in Corollary~\ref{cor:upper-tail-rate} for $\Omega = \Z/N\Z$ if $N\to\infty$ along some sequence other than the primes.
\end{remark}

\subsection{Microscopic scale}
\label{sec:statements-micro}

In the microscopic scale, $\delta^{-3} p^{k-2}(\log(1/p))^2 \to \infty$. Here $\delta = \delta_N$ is thought as tending to zero relatively quickly compared to $p = p_N$.
We will establish, in Section~\ref{sec:micro}, the following asymptotic solution to the variational problem \eqref{eq:var} in this regime.

\begin{theorem} \label{thm:micro-rate}
Fix $k\ge 3$. Let $\Omega = \Z / N\Z$ or $[N]$. Let $p = p_N \in (0,1)$ and $\delta = \delta_N > 0$ be such that $p \to 0$ and $\delta^{-3} p^{k-2}(\log(1/p))^2 \to \infty$. Then, as $N \to \infty$,
\[
\phi_p^{(k, \Omega)}(\delta) =
\begin{cases}
\left(\cfrac{1}{2k^2}+o(1)\right) \delta^2 Np & \text{if } \Omega = \Z/N\Z, \\
\left(\cfrac{1}{2\gamma_k}+o(1)\right) \delta^2 Np & \text{if } \Omega = [N],
\end{cases}
\]
where
\begin{equation}
	\label{eq:gamma}
\gamma_k
= \frac{4}{3}\left(k + \sum_{0\le i < j < k}
\frac{(k-1)^2-i^2-(k-1-j)^2}{(k-1-i)j}\right).
\end{equation}
\end{theorem}

\begin{remark}
The first few values of $\gamma_k$ are $\gamma_3= 28/3$, $\gamma_4 = 17$, and $\gamma_5 = 718/27$. We are not aware of a closed-form expression for $\gamma_k$. However, one always has $\gamma_k \ge k^2$, and asymptotically $\lim_{k\to \infty} \gamma_k/k^2 = (30 - 2\pi^2)/9 \approx 1.14$.
\end{remark}

Combining with the large deviation principle Theorem~\ref{thm:ldpkap}, we obtain the following corollary.

\begin{corollary} \label{cor:micro-prob}
Fix $k \ge 3$. Let $\Omega = [N]$ or $\Z/ N \Z$, and in the latter case assume that $N$ is prime. 
Let $p = p_N \to 0$ and $\delta = \delta_N > 0$ be such that $\delta = O(1)$, $\delta^{-3} p^{k-2}(\log(1/p))^2 \to \infty$, and
\[
\min\{\delta p^k,\delta^2 p\} \ge N^{-\frac{1}{6(k-1)}} \log N
\]
(for $k=3$, the right-hand side can be relaxed to $N^{-1/6}(\log N)^{7/6}$; for $k=4$, it can be relaxed to $N^{-1/12} (\log N)^{13/12}$).
Then, as $N \to \infty$, the random variable $X_k = T_k(\Omega_p)$ satisfies
\[
\frac{- \log \PP\left(X_k \ge (1+\delta)\EE X_k\right)}{\delta^2 N p}
= \begin{cases}
  1/(2k^2) +o(1)
  & \text{if } \Omega = \Z/N\Z, \\
  1/(2\gamma_k) + o(1)
  & \text{if } \Omega = [N],
 \end{cases}
\]
where $\gamma_k$ is defined in \eqref{eq:gamma}.
\end{corollary}

The upper bound on $\phi_p^{(k,\Omega)}$ in Theorem~\ref{thm:micro-rate} for $\Omega = \Z/N\Z$ is obtained by taking the constant function on $\Omega$ with value $p(1+ \delta)^{1/k}$, which turns out to be tight asymptotically to the first order. This behavior, where the solution to the variational problem is obtained by a constant function, at least asymptotically, suggests that the reason for many $k$-APs in the microscopic scale is a uniform boost in the density of the set, and such phenomena are referred to in the literature as \emph{replica symmetry}~\cite{CV11}. (Admittedly we are somewhat abusing terminology here, as replica symmetry in previous works \cite{CV11,LZ15,Zhao-lower} on random graphs refer to setting of constant $p$ and $\delta$). In Section~\ref{sec:replica} we record some partial results on replica symmetry for constant $p$ and $\delta$ for $k$-APs. On the contrary, in the macroscopic scale, many $k$-APs are created by a smaller set arranged in a special structure, for example, an interval, and this is referred as \emph{replica symmetry breaking}.

When $\Omega=[N]$, in the microscopic scale, the asymptotically optimal solution to the variational problem  turns out not to be a constant function, but rather, a function that assigns each $a \in [N]$ to a number proportional to the number of $k$-APs in $[N]$ containing $a$. This is due to the asymmetry of the elements in $[N]$, as the elements in the middle bulk are contained in more $k$-APs than those in the fringe. Even though the constant function does not asymptotically minimize the variational problem in this setting, the solution nevertheless exhibits some features of replica symmetry (by analogy to the $\Z/N\Z$ setting). We find this new phenomenon interesting, as we are not aware of analogous results in the random graph setting.

\section{Gaussian width and non-linear large deviations}
\label{sec:ldpkap}

In this section we apply Eldan's non-linear large deviation principle \cite{Eldan} to $k$-AP upper tails, reducing the large deviation rate problem to a variational problem. The proof relies on bounding the Gaussian width of a set of gradients, which will be done in Section~\ref{sec:gw}.

\subsection{Eldan's LDP}
\label{sec:ldpgw}

We start with a short discussion of Eldan's \cite{Eldan} result (adapted to our setting). For any $K \subset \R^N$ define the \emph{Gaussian width} of $K$ by
\[
\GW(K) := \EE \bigl[ \sup_{x \in K} \anglb{x, Z} \bigr]
\]
where the expectation is taken over $Z \sim \Normal(0,I_N)$, a standard Gaussian random vector in $\RR^N$.

For any function $F \colon \{0,1\}^N \to \RR$, define its discrete derivatives by
\[
  \DiscDerv_i F(y) =
     F(y_1, \dots, y_{i-1}, 1, y_{i+1}, \dots, y_N)
     - F(y_1, \dots, y_{i-1}, 0, y_{i+1}, \dots, y_N)
\]
for any $i \in [N]$ and $y=(y_1, y_2,\dots, y_N) \in \{0, 1\}^N$, and its discrete gradient by
\[
  \DiscGrad F (y)= (\DiscDerv_1 F (y), \dots, \DiscDerv_N F(y)).
\]
A key quantity is the Gaussian width of the set of all discrete gradients of $F$, which we denote by
\begin{equation} \label{eq:DiscGW}
	\DiscGW(F) := \GW \left(\left\{\DiscGrad F(y) : y \in \{0,1\}^N \right\}\cup \{0\}\right).
\end{equation}
Define the \emph{discrete Lipschitz constant} of $F$ by
\[
\DLip(f)=\max_{i \in [N], y \in \{0, 1\}^N} \DiscDerv_i f(y).
\]
 
Improving an earlier result of Chatterjee and Dembo~\cite{CD},  Eldan \cite{Eldan} proved a large deviation principle for general non-linear functions $F\colon \{0, 1\}^N \rightarrow \R$ satisfying certain hypotheses on its set of discrete gradients. The large deviation rate is given in terms of the natural variational problem:
\begin{equation} \label{eq:eldan-var}
	\varphi_p^F(t) := \inf_{y \in [0, 1]^N}\left\{ \sum_{i=1}^N I_p(y_i):  \E F(Y)\ge tN \right\},	
\end{equation}
where the expectation is taken with respect to a random vector $Y = (Y_1, Y_2, \dots, Y_N)$ with $Y_i \sim \Bernoulli(y_i)$ independently for every $i \in [N]$.

\begin{theorem}[Eldan \cite{Eldan}] \label{thm:eldan} Let $X = (X_1, X_2, \dots, X_N) \in \{0,1\}^N$ be a random vector with i.i.d.\ $X_i \sim \Bernoulli(p)$. Given a function $F\colon \{0, 1\}^N \to \RR$, for every $t,\varepsilon \in \R$ with $0<\varepsilon < \varphi_p^F(t-\varepsilon)/N$, we have
\begin{align*}
	\log \P(F(X) \ge t N) \le - \varphi_p^F(t-\varepsilon)\left(1-\frac{6L (\log N)^{1/6}}{N^{1/3}}\right)
\end{align*}
with 
\begin{equation} \label{eq:GWbound}
	L = \frac{1}{\varepsilon} \left(2 {\DLip}(F)+\frac{1}{\varepsilon \sqrt{N}}{{\DLip}(F)} + |\log (p(1-p))| \right)^{2/3} \left( \DiscGW(F) + \frac{1}{\varepsilon} \DLip(F)^2 \right)^{1/3}.
\end{equation}
Moreover, whenever the assumption $\frac{2}{N\varepsilon^2} \DLip (F)^2 \leq \frac{1}{2}$ holds, the following lower bound holds: 
\begin{align*}
\log \P(F(X)\geq (t-\varepsilon)N) \geq -\varphi_p^F(t)\left(1+\frac{2}{N\varepsilon^2}  \DLip(F)^2 \right) -\log 10.
\end{align*}
\end{theorem}

Denote the usual gradient and partial derivatives of $F \colon \RR^N \to \RR$ by
\[
\nabla F := (\partial_1 F, \dots, \partial_N F).
\]
Define 
\begin{equation}
	\label{eq:GW-F}
\GW(F) := \GW \left( \left\{\nabla F(y) \colon y \in \{0,1\}^N \right\} \cup \{0\} \right)
\end{equation}
for the continuous analog of $\DiscGW(F)$ from \eqref{eq:DiscGW}. These two quantities differ only negligibly in our applications.

\begin{lemma} \label{lem:discGW-GW}
	For any twice-differentiable $F \colon \RR^N \to \RR$, we have
	\[
	|\DiscGW(F) - \GW(F)| \lesssim \sum_{i=1}^N \sup_{x \in [0,1]^N} |\partial_{ii} F (x)|,
	\]
	where $\partial_{ij} F = \partial^2 F/\partial x_i \partial x_j$ denotes the $(i,j)$-th partial derivative of $F$ and $\DiscGW(F)$ is defined by considering the restriction of $F$ to $\{0,1\}^N$.
\end{lemma}

\begin{proof}
	Applying the intermediate value theorem (twice), we have
	\[
	\absb{\DiscDerv_i F(y) - \partial_i F(y)} \le \sup_{x \in [0,1]^N} |\partial_{ii} F (x)|
	\]
	for any $y \in \{0,1\}^N$. Thus for any $Z = (Z_1, \dots, Z_N) \in \RR^N$ and any $y \in \{0,1\}^N$,
	\begin{align*}
		|\anglb{\DiscGrad F(y), Z} - \anglb{\nabla F(y), Z}|
		\le \sum_i \sup_{x \in [0,1]^N} |\partial_{ii} F (x)| |Z_i|.
	\end{align*}
	The result follows by first taking the supremum over $y$, and then taking an expectation over $Z \sim \Normal(0, I_N)$ and using $\EE |Z_i| = O(1)$.
\end{proof}

\subsection{LDP for $k$-AP}

Now, we apply Theorem~\ref{thm:eldan} to derive a large deviation principle for $k$-AP counts, conditioned on bounds for the Gaussian width of the gradients of the $k$-AP counting function.

By viewing points in $\RR^\Omega$ as functions $\Omega \to \RR$, the previously defined $k$-AP functional $T_k$ can be viewed as a function on $\RR^{\Omega}$ by
\[
T_k(y)= \sum_{a,b} y_a y_{a+b} y_{a+2b} \dots y_{a+(k-1)b}, \quad y \in \RR^\Omega.
\]
In the case $\Omega = [N]$, the indices $a$ and $b$ both range over $\ZZ$, and we set $y_a = 0$ if $a \notin [N]$. In the case $\Omega = \ZZ/N\ZZ$, the indices $a$ and $b$ both range over $\ZZ/N\ZZ$ and the indices of $y$ are taken mod $N$. Recall the definition~\eqref{eq:var} of the variational problem for upper tails of $k$-AP counts, reproduced here:
\begin{equation}
	\label{eq:var-kap-pt}
	\phi_p^{(k,\Omega)}(\delta) :=  \inf_{y \in [0,1]^\Omega} \left\{ \sum_{a \in \Omega} I_p(y_a):  T_k(y) \ge (1+\delta)p^k T_k(\Omega) \right\}.
\end{equation}
We will apply Theorem~\ref{thm:eldan} to the function $F = T_k/N$. The relevant Gaussian width is
\[
\GW(T_k/N) = \GW \left( \left\{ \tfrac{1}{N}\nabla T_k(y) \colon y \in \{0,1\}^\Omega \right\} \right).
\]
Note that we do not need to include the origin in the definition since $\nabla T_k(0) = 0$. We have the ``trivial'' bounds:
\begin{equation}
	\label{eq:gw-trivial-bounds}
\sqrt{N} \lesssim \GW(T_k/N) \lesssim N.
\end{equation}
The lower bound comes from considering the constant vector $y = (1, \dots, 1)$, and the upper bound comes from noting that $\frac{1}{N} \nabla T_k(y)$ is coordinatewise $O(1)$ for all $y \in \{0,1\}^\Omega$.

The main result of this section is the following proposition, showing that any power-saving improvement to the trivial upper bound to $\GW(T_k/N)$ leads to a large deviation principle allowing the probability $p$ to decay as $N^{-c}$. Combining it with bounds on the Gaussian width (to be proved in the next section) gives Theorem~\ref{thm:ldpkap}.

\begin{proposition}\label{prop:ap-gw-ldp}
Fix $k \ge 3$. Suppose we have real constants $\sigma, \tau$ such that
\begin{equation}
	\label{eq:hypothesis-gw-bound}
\GW(T_k/N) = O(N^{1-\sigma} (\log N)^\tau).
\end{equation}
Let $p=p_N$ be bounded away from 1, and $\delta = \delta_N > 0$ be such that 
$\delta=O(1)$
and 
\begin{equation}
	\label{eq:p-lower-sigma-tau}
N^{-\sigma/3} (\log N)^{\tau/3+1} \lesssim \min\{\delta p^k , \phi_p^{(k,\Omega)}(\delta/2)/N \}.
\end{equation}
Then
\[
- \log \PP(T_k(\Omega_p) \ge (1 + \delta)\EE T_k(\Omega_p)) \\
= (1+o(1))\phi_p^{(k,\Omega)}(\delta + o(\delta)).
\]
\end{proposition}

As long as one can prove an estimate~\eqref{eq:hypothesis-gw-bound} on Gaussian width with $\sigma > 0$, we can allow $p$ to decay as $N^{-c}$ for some constant $c > 0$. 
In Theorem~\ref{thm:gw-bound}, we show that one can take $\sigma = 1/(2(k-1))$ and $\tau = 0$ (with better bounds for $k=3,4$). From the asymptotic solutions to the variational problems (Theorem~\ref{thm:rate} and \ref{thm:micro-rate}), and noting that $\phi_p^{(\Omega,k)}(\delta)$ is monotonic in $\delta$, we have, as long as $\delta p^k N^2 \to \infty$,
\[
\phi_p^{(\Omega,k)}(\delta/2)/N \asymp \min\{\sqrt{\delta} p^{k/2} \log(1/p), \delta^2 p\}.
\]
Combining these asymptotics, we see that hypothesis~\eqref{eq:p-lower-sigma-tau} translates into hypothesis~\eqref{eq:hyp-p-lower} in Theorem~\ref{thm:ldpkap}, and hence Proposition~\ref{prop:ap-gw-ldp} implies Theorem~\ref{thm:ldpkap}.

\medskip

In the rest of this section, we prove Proposition~\ref{prop:ap-gw-ldp}.
We first prove some easy estimates on the various quantities that appear in Theorem~\ref{thm:eldan}. Recall the definitions of $\DiscGW(F)$ and $\GW(F)$ from \eqref{eq:DiscGW} and \eqref{eq:GW-F}.

\begin{lemma} \label{lem:DLip-DiscGW-T_k-bound}
	For any $k \ge 3$, we have
	\[
		\DLip(T_k/N) = O(1)
	\]
	and
	\[
		\DiscGW(T_k/N) = \GW(T_k/N) + O(1).
	\]
\end{lemma}

\begin{proof}
	The first claim follows from noting that in $T_k(y) = \sum_{a,b} y_a y_{a+b} \cdots y_{a+(k-1)b}$, every variable $y_a$ appears in $O(N)$ terms.
	
	The second claim follows from Lemma~\ref{lem:discGW-GW}, as $\partial_{aa} T_k(y)$ is uniformly bounded for all $y \in [0,1]^\Omega$ and all $a \in \Omega$.
\end{proof}

\begin{proof}[Proof of Proposition~\ref{prop:ap-gw-ldp}] We apply Theorem \ref{thm:eldan} for $F = T_k/N$. Set
\[
\varepsilon = N^{-1/3} (\log N)^{11/12}\GW(T_k/N)^{1/3} \lesssim N^{-\sigma/3} (\log N)^{\tau/3 + 11/12}.
\]
By \eqref{eq:p-lower-sigma-tau},
\begin{equation}\label{eq:eps-small}
\varepsilon  =  o(\delta p^k) \quad \text{and} \quad 
\varepsilon = o(\phi_p^{(k,\Omega)}(\delta/2)/N).
\end{equation}

Note that $\sigma \le 1/2$ due to the lower bound $\GW(T_k/N) \gtrsim \sqrt{N}$ in \eqref{eq:gw-trivial-bounds}. So in particular, $N \varepsilon^2 \to \infty$. Also, $\log(1/p) = O(\log N)$ by \eqref{eq:p-lower-sigma-tau}.

Recall $L$ from \eqref{eq:GWbound}. Using Lemma~\ref{lem:DLip-DiscGW-T_k-bound} and earlier estimates, we have
\[
L \lesssim \varepsilon^{-1} (\log N)^{2/3} \GW(T_k/N)^{1/3}.
\]
Thus, as $N \to \infty$,
\[
L N^{-1/3}(\log N)^{1/6} 
\lesssim
\varepsilon^{-1} N^{-1/3} (\log N)^{5/6} \GW(T_k/N)^{1/3}
=
(\log N)^{-1/12} \to 0.
\]

Let $Y = (Y_1, \dots, Y_N)$ be a random vector with $Y_i \sim \Bernoulli(y_i)$ independently for all $i \in [N]$. Then
\[
\EE T_k(Y) = T_k(y) + O(N) = T_k(y) + o(\delta p^k N^2).
\]
The discrepancy $O(N)$ comes from terms in $T_k(y)$ where some $y_i$ may appear more than once. Setting 
\[
t = (1+\delta)p^k T_k(\Omega)/N^2,
\]
we see that
\[
\EE T_k(Y)/N \ge (t + o(\delta p^k)) N \quad \text{is equivalent to} \quad T_k(y) \ge (1+\delta + o(\delta))p^kT_k(\Omega).
\]
Comparing the definition of  $\varphi^F_p(t)$ from \eqref{eq:eldan-var} for $F = T_k/N$ and $\phi^{(k, \Omega)}_p(\delta)$ from \eqref{eq:var-kap-pt} (and noting that $\varphi^F_p(t)$ is non-decreasing in $t$ and $\phi^{(k,\Omega)}_p(\delta)$ is non-decreasing in $\delta$), we obtain
\[
\varphi^{T_k/N}_p (t \pm \varepsilon) = \phi_p^{(k, \Omega)}(\delta + o(\delta)).
\]
The hypothesis $0 < \varepsilon < \tfrac{1}{N} \varphi_p^{T_k/N}(t - \varepsilon)$ in Theorem~\ref{thm:eldan} is also satisfied due to \eqref{eq:eps-small}.

Applying Theorem~\ref{thm:eldan}, we obtain the upper bound to the log-probability
\[
\log \PP(T_k(\Omega_p) \ge (1+\delta)p^k T_k(\Omega))
\le -(1-o(1)) \varphi^{T_k/N}_p (t - \varepsilon) \sim - \phi_p^{(k,\Omega)}(\delta - o(\delta)),
\]
as well as the lower bound (changing $t$ to $t + \varepsilon$ when applying Theorem~\ref{thm:eldan}),
\[
\log \PP(T_k(\Omega_p) \ge (1+\delta)p^k T_k(\Omega))
\ge -(1-o(1)) \varphi^{T_k/N}_p(t+\varepsilon) - O(1)
\sim - \phi_p^{(k,\Omega)}(\delta + o(\delta)).
\]
Combining the upper and lower bounds, and recalling that $\EE T_k(\Omega_p) \sim p^k T_k(\Omega)$, the result follows. 
\end{proof}

\section{Bounds on Gaussian width}
\label{sec:gw}

In this section, we establish bounds on the Gaussian width of the set of gradients of $T_k$. These bounds can be used in Proposition~\ref{prop:ap-gw-ldp} from the previous section to deduce Theorem~\ref{thm:ldpkap} on the LDP for APs. 

Our main result of this section is stated below. Recall from Section~\ref{sec:ldpgw} that
\begin{align}
\GW(T_k /N)
&= \GW(\{ \nabla T_k(y)/N : y \in \{0,1\}^\Omega\}) \nonumber
\\
&= \EE_{Z \sim \Normal(0, I_N)} \sup_{y \in \{0,1\}^\Omega} \anglb{\tfrac{1}{N} \nabla T_k(y), Z}. \label{eq:GW-Tk}
\end{align}

\begin{theorem} \label{thm:gw-bound}
For any fixed $k \ge 3$,
\begin{align}\label{eq:Tkgradient}
 \GW(T_k/N) = O(N^{1 - \frac{1}{2(k-1)}}).
\end{align}
Furthermore, for $k = 3$, the bound can be tightened to 
\[
\GW(T_3/N) = \Theta(\sqrt{N \log N}).
\] 
For $k = 4$, the bound can be improved to 
\[
\GW(T_4/N) = O( N^{3/4} (\log N)^{1/4}).
\]
\end{theorem}

\begin{remark}
After a preprint of our paper had appeared, Bri\"et and Gopi~\cite{BG} improved the  bound to $\GW(T_k/N) = O(N^{1 - \frac{1}{2\lceil (k-1)/2\rceil}}(\log N)^{1/2})$ for all $k \ge 5$.
\end{remark}

We have an easy lower bound $\GW(T_k/N) \gtrsim \sqrt{N}$ deduced by taking the constant vector $y = (1, \dots, 1)$ in~\eqref{eq:GW-Tk}. We conjecture that it is essentially tight.

\begin{conjecture} \label{conj:gw-T_k}
For any fixed $k \ge 4$,
\[
\GW(T_k/N) = \sqrt{N} (\log N)^{O(1)}.
\]
\end{conjecture}

For the proof of Theorem~\ref{thm:gw-bound}, we go back to viewing $T_k$ as an operator on functions $\Omega \to \RR$ (as opposed to a function on points in $\RR^\Omega$). We define a multilinear version of $T_k$ by setting, for $f_0, \dots, f_{k-1} \colon \Omega \to \RR$,
\[
T_k(f_0, \dots, f_{k-1}) := \sum_{a,b} f_0(a)f_1(a+b) \cdots f_{k-1}(a+(k-1)b),
\]
so that for $f \colon \Omega \to \RR$,
\[
T_k(f) = T_k(\underbrace{f, f, \dots, f}_{k \text{ times}}).
\]
 
We identify points in $\CC^\Omega$ with functions $\Omega \to \CC$ and maintain the notation $\langle \ , \ \rangle$ for inner products, so that for $f,h \colon \Omega \to \CC$,
\[
\anglb{f,h} := \sum_{a \in \Omega} f(a)\ol{h(a)}.
\]
The gradient $\nabla T_k$ of $T_k$ maps a function $f \colon \Omega \to \RR$ to the function $\nabla T_k(f) \colon \Omega \to \RR$ defined by
\[
\nabla T_k(f)(a) =\sum_{b} \sum_{i=0}^{k-1} \prod_{\substack{0 \leq j \leq k-1 \\ j \neq i}} f(a + (j-i)b), \quad a \in \Omega.
\]
Since $T_k$ is multi-linear, for any $f,h \colon \Omega \to \RR$, 
\begin{equation}
	\label{eq:grad-decomp}
	\anglb{\nabla T_k(f),h}
	= \sum_{j =0}^{k-1} T_k(\underbrace{f,f,\dots,f}_{j \text{ times}}, h, \underbrace{f,f,\dots,f}_{k-1-j \text{ times}}).
\end{equation}

\subsection{3-APs and Fourier analysis} \label{sec:grad-3ap} 

Here, we prove the claim in Proposition~\ref{prop:T_k-hull-size} for $k=3$ using Fourier analysis. It will be easier to work in the setting $\Omega = \ZZ/N\ZZ$. The corresponding bounds for $\Omega = [N] \subset \ZZ$ can be easily derived by embedding $[N]$ in $\ZZ/N'\ZZ$ for some larger $N' \in [2N, 3N]$ so that $3$-APs in $[N] \subset \ZZ/N'\ZZ$ do not wrap-around zero in the cyclic group.

Given $f \colon \ZZ/N\ZZ \to \CC$, define its discrete Fourier transform by
\[
\wh f(r) = \frac{1}{N} \sum_{a \in \ZZ/N\ZZ} f(a) \omega^{-ar}, \quad r \in \ZZ/N\ZZ,
\]
where $\omega = e^{2\pi i/N}$. The inverse transform is given by
\[
 f(a) = \sum_{r \in \ZZ/N\ZZ} \wh f(r) \omega^{ar}. 
\]

The following standard identity relates $T_3$ with the Fourier transform.

\begin{lemma} \label{lem:3ap-fourier-id}
	For $f,g,h \colon \ZZ/N\ZZ \to \RR$,
	\[
	\frac{1}{N^2} T_3(f,g,h) = \sum_{r \in \ZZ/N\ZZ} \wh f(r) \wh g(-2r) \wh h(r).
	\]
\end{lemma}

\begin{proof}
Expanding the left-hand side using the inverse transform, we have
\begin{align*}
T_3(f,g,h) 
&=
\sum_{a,b \in \ZZ/N\ZZ} f(a)g(a+b)h(a+2b)
\\
&= 
\sum_{a,b \in \ZZ/N\ZZ} \sum_{r,s,t \in \ZZ/N\ZZ} \wh f(r)\wh g(s) \wh h(t) \omega^{ar + (a+b)s + (a+2b)t}
\\
&= N^2 \sum_{r \in \ZZ/N\ZZ} \wh f(r) \wh g(-2r) \wh h(r),
\end{align*}
where the final step follows from noting that
\[
\sum_{a,b \in \ZZ/N\ZZ}\omega^{ar + (a+b)s + (a+2b)t}
=
\begin{cases}
N^2 & \text{if } r + s + t = 0 \text{ and } s+ 2t = 0, \\
0 & \text{otherwise.} 
\end{cases}
\]
\end{proof}

The above identity leads to the following bound, showing that $T_3$ is controled by the Fourier transform of its inputs.

\begin{lemma} \label{lem:3ap-fourier-bound}
	Let $f_0,f_1,f_2 \colon \ZZ/N\ZZ \to [-1,1]$. For each $i = 0, 1, 2$,
	\[
	\frac{1}{N^2} |T_3(f_0,f_1,f_2)| \lesssim \|\wh f_i\|_\infty =: \max_r |\wh f_i(r)|.
	\]
\end{lemma}

\begin{remark}
If $N$ is odd, then the $\lesssim$ can be replaced by $\le$.	
\end{remark}

\begin{proof}
	Using Lemma~\ref{lem:3ap-fourier-id} and the Cauchy--Schwarz inequality, we have
	\begin{align*}
		\frac{1}{N^2} |T_3(f_0,f_1,f_2)|
		&\le
		\|\wh f_0\|_\infty \sum_{r \in \ZZ/N\ZZ} |\wh f_1(-2r)||\wh f_2(r)|
		\\
		&\le 
		\|\wh f_0\|_\infty \bigl(\sum_r |\wh f_1(-2r)|^2\bigr)^{1/2} \bigl(\sum_r |\wh f_2(r)|^2 \bigr)^{1/2}
		\\
		&\lesssim
		\|\wh f_0\|_\infty \bigl(\sum_r |\wh f_1(r)|^2\bigr)^{1/2} \bigl(\sum_r |\wh f_2(r)|^2 \bigr)^{1/2}
		\\
		&\le \|\wh f_0\|_\infty.
	\end{align*}
	The final step follows from Parseval's identity: $\sum_r |\wh f_j(r)|^2 = \frac1N \sum_a |f_j(a)|^2 \le 1$. The proofs for $i = 1,2$ are analogous.
\end{proof}

\begin{proof}[Proof of Theorem~\ref{thm:gw-bound} for $k=3$]
	For any $g \colon \ZZ/N\ZZ \to \RR$, we have
	\[
	\anglb{\nabla T_3(f), g} = T_3(g,f,f) + T_3(f,g,f) + T_3(f,f,g),
	\]
	so by Lemma~\ref{lem:3ap-fourier-bound},
	\[
	\frac{1}{N} |\anglb{\nabla T_3(f), g}| \lesssim N \max_r |\wh g(r)|.
	\] 
	Now let $g$ be a random function taking i.i.d.~standard normal values. Then each $\wh g(r)$ is a normally distributed complex number with $\EE [|\wh g(r)|^2] = 1/N$, since the Fourier transform is a unitary operator. Standard results about the supremum of Gaussian processes, e.g., Lemma~\ref{lem:small-set-small-gw} below, then gives $\EE \sup_r |\wh g(r)| \lesssim \sqrt{(\log N)/N}$. Thus
	\[
	\GW(T_3/N) = \EE \sup_{f \colon \ZZ/N\ZZ \to \{0,1\}} \frac{1}{N} \anglb{\nabla T_3(f), g}
	\lesssim N \EE \sup_r |\wh g(r)| \lesssim \sqrt{N \log N}.
	\]
	The matching lower bound is proved in Appendix~\ref{app:gw}. 
\end{proof}

See Appendix~\ref{app:gw} for proof of the bound $\GW(T_4/N) = O( N^{3/4} (\log N)^{1/4})$ in Theorem~\ref{thm:gw-bound}, which extends the above Fourier analytic technique.

\subsection{$k$-APs and the Chinese remainder theorem}

Our strategy for proving Theorem~\ref{thm:gw-bound} is to show that the set of gradients $\nabla T_k(f) /N$ over all $f \colon \Omega \to \{0,1\}$ is contained in the convex hull of a small set of bounded functions. We start with a standard bound on Gaussian width of sets. The proof is included for completeness.

\begin{lemma}[Small sets have small Gaussian widths] \label{lem:small-set-small-gw}
If $S \subset [-1,1]^N$, then 
\[ 
\GW(S) = O(\sqrt{N \log |S|}).
\]
\end{lemma}

\begin{proof}
We may assume that $|S| \ge 2$.
We have the following standard tail bound for $Z \sim \Normal(0,\sigma^2)$: $\PP(Z > t) \le e^{-t^2/(2\sigma^2)}$. With $Z \sim \Normal(0,I_N)$ a standard Gaussian vector in $\RR^N$, one has $\anglb{y,Z} \sim \Normal(0,|y|^2)$ for every $y \in \RR^N$. Thus $\PP(\anglb{y,Z} > t) \le e^{-t^2/(2|y|^2)} \le e^{-t^2/(2N)}$ for every $y \in [-1,1]^N$. So for any $t > 0$, we have by the union bound
\[
\PP(\sup_{y \in S} \anglb{y,Z} > t) \le |S| e^{-t^2/(2N)}.
\]
Set $u = 10\sqrt{N \log |S|}$. Then
\begin{align*}
\GW(S) = \EE_Z \sup_{y \in S} \anglb{y,Z} 
&\le u + \int_u^\infty \PP(\sup_{y \in S} \anglb{y,Z} > t) \, dt 
\\
&\le u + \int_u^\infty |S| e^{-t^2/(2N)} \, dt
= O(\sqrt{N \log |S|}). \qedhere
\end{align*}
\end{proof}

\begin{corollary} \label{cor:T_k-hull}
If $\{\nabla T_k(f) /N : f \colon \Omega \to \{0,1\}\}$ is contained in the convex hull of a set $S$ of uniformly bounded functions\footnote{Here, and in the sequel, uniformly bounded functions refer to those functions $f$ with $|f(a)| = O(1)$ point-wise.} on $\Omega$, then $\GW(T_k/N) = O(\sqrt{N \log |S|})$.
\end{corollary}

We will prove the following bounds on the size of $S$, which imply the bound $\GW(T_k/N) = O(N^{1-\frac{1}{2(k-1)}})$ in Theorem~\ref{thm:gw-bound}.

\begin{proposition} \label{prop:T_k-hull-size}
	Fix $k \ge 3$. The set 
	$\{\nabla T_k(f) /N : f \colon \Omega \to \{0,1\}\}$ is contained in the convex hull of a set $S$ of uniformly bounded functions on $\Omega$, where $\abs{S} = \exp(O(N^{1-1/(k-1)}))$.
\end{proposition}

We can conjecture that the bound on $\abs{S}$ can be improved to $e^{N^{o(1)}}$, or perhaps even the following stronger bound, which would imply Conjecture~\ref{conj:gw-T_k}.

\begin{conjecture}
	In Proposition~\ref{prop:T_k-hull-size}, for every fixed $k \ge 4$, one can have $\abs{S} = e^{(\log N)^{O(1)}}$.
\end{conjecture}

\begin{remark}
For $k = 3$, one can have $\abs{S} = O(N)$ by considering the Fourier basis and using Lemma~\ref{lem:3ap-fourier-bound}. For the first open case $k=4$, intuition from the theory of \emph{higher-order Fourier analysis} (e.g., \cite{TaoHOFA}) suggests that perhaps it suffices to take the set of ``quadratic characters'', i.e., functions of the form $a \mapsto e^{2\pi i q(a)}$ where $q(a)$ behaves quadratically (very loosely speaking). More generally, perhaps we can take a set of \emph{nilsequences}~\cite{GT-nilsequence}. However, this method is currently incapable of proving the conjecture due to poor quantitative dependencies in the theory of higher order Fourier analysis. See~\cite[Section 4.1]{FRA} for the analogous problem in the language of ergodic averages.
\end{remark}

For the proof of Proposition~\ref{prop:T_k-hull-size}, it will be easier to work in the setting $\Omega = [N]\subset \ZZ$. The proof can be easily adapted to work for $\Omega = \ZZ/N\ZZ$ by chopping $\ZZ/N\ZZ$ into a bounded number of intervals and analyzing their contributions separately. 

The argument we present here is essentially the same as that appearing in~\cite[Section 4]{FLW}, which in turn is motivated by the random sampling idea in~\cite{CG16}. Roughly speaking, we can obtain a good estimate of $\nabla T_k(f)$ by knowing $f$ only on $A \subset [N]$, where $A$ is a random subset of slightly more than $N^{1-1/(k-1)}$ elements. Thus any gradient $\nabla T_k(f)$ can be well approximated by one of $2^{\abs{A}}$ many possible functions.\footnote{This random sampling argument can also be used in conjunction with Chatterjee and Dembo's results~\cite{CD} to obtain a large deviation principle for $k$-APs, allowing the probability $p$ to decay as $N^{-c_k}$ for some constant $c_k > 0$, though with a smaller $c_k$ than what is shown here.} Actually one can achieve a similar effect (more simply and with better bounds) deterministically by partitioning $[N]$ modulo $k-1$ distinct primes of sizes roughly $N^{1/(k-1)}$ each, which is how we shall proceed. In this approach, the Chinese remainder theorem play a role similar to independence for random variables.

\begin{proof}[Proof of Proposition~\ref{prop:T_k-hull-size}]
We may assume that $N$ is large.
To analyze $\nabla T_k(f)$, we make the following definition. For $c_1,\dots,c_{k-1} \in \ZZ$, let $\Phi(c_1,\dots,c_{k-1})$ be the set of all functions $F \colon [N] \to \RR$ of the form
\begin{equation}\label{eq:za} 
F(a) =  \frac{1}{N} \sum_{b \in \ZZ} f(a+c_1b) f(a+c_2b) \dots f(a+c_{k-1}b), \quad a \in [N]
\end{equation}
for some $f \colon [N] \to \{0,1\}$ (set $f(a) = 0$ if $a \notin [N]$). It suffices to show that $\Phi(c_1,\dots,c_{k-1})$ is contained in the convex hull of $\exp(O(N^{1-1/(k-1)}))$ many bounded functions, whenever $\{c_1,\dots,c_{k-1}\}$ is any of
\[ \{1,2,\dots,k-1\}, \{-1,1,\dots,k-2\},\dots,\{-(k-1),\dots,-2,-1\}. \]
For clarity, we assume that $(c_1,\dots, c_k) = (1, \dots, k-1)$, as the other cases are similar. 

Pick $k-1$ distinct primes $q_1,\dots,q_{k-1} \in [(N/2)^{\frac{1}{k-1}}, N^{\frac{1}{k-1}}]$ (such a choice always exist when $N$ is large due to bounds on large gaps between primes~\cite{BHP01}). We have $q_1q_2 \cdots q_{k-1} \in [N/2,N]$. Write
\begin{equation} \label{eq:parition-f-mod}
f = \sum_{r_i \in \ZZ/q_i\ZZ} f_{r_i + q_i\ZZ}
\end{equation}
where
\[ 
   f_{r_i + q_i\ZZ}(a) = 
   \begin{cases} 
      f(a) & \text{if }a \equiv r_i\pmod{q_i}, \\ 
      0 & \text{if }a \not\equiv r_i\pmod{q_i}, 
   \end{cases} 
\]
is the restriction of $f$ to the residue class $r_i \pmod{q_i}$. By using  \eqref{eq:parition-f-mod} to partition the $i$-th factor $f(a+ib)$ in each term in \eqref{eq:za}, we obtain
\[	
 F(a) = \frac{1}{N} \sum_{b \in \ZZ} \prod_{i=1}^{k-1} \sum_{r_i \in \ZZ/q_i\ZZ} f_{r_i + q_i\ZZ}(a+ i b)
 = \frac{1}{q_1\cdots q_k} \sum_{r_1, \dots, r_{k-1}} v_{r_1, \dots, r_{k-1}; f}(a)
\]
where the sum is taken over $r_1 \in \ZZ/q_1\ZZ, \dots, r_{k-1} \in \ZZ /q_{k-1}\ZZ$, and 
\begin{equation}\label{eq:vertices} 
v_{r_1,\dots,r_{k-1};f}(a) := \frac{q_1q_2\cdots q_{k-1}}{N} \sum_{b \in \ZZ} \prod_{i=1}^{k-1} f_{r_i+ q_i\ZZ}(a+ib).
\end{equation}
We see that $F$ lies in the convex hull of 
\[ V = \{v_{r_1,\dots,r_{k-1};f} : \text{$f$ is $\{0,1\}$-valued}, r_1 \in \ZZ/q_1\ZZ, \dots, r_{k-1} \in \ZZ/q_{k-1}\ZZ\}. \]
For fixed $i$ and $r_i$, the number of possibilities for $f_{r_i + q_i\ZZ}$ is at most $2^{N/q_i+1}$ as $f$ ranges over $\{0,1\}$-valued functions. 
It follows that 
\[ |V| \leq q_1q_2 \cdots q_{k-1} \prod_{i=1}^{k-1} 2^{N/q_i+1} = 2^{O(N^{1-1/(k-1)})}, \]
by our choice of $q_i$. It remains to show that all functions in $V$ are bounded. Indeed, for any fixed $a \in [N]$, the simultaneous congruence conditions $a+ib \equiv r_i\pmod{q_i}$, for $1 \leq i \leq k-1$, in the unknown $b \in [-N,N]$ have $O(N/(q_1q_2\cdots q_{k-1})+1)$ solutions by the Chinese Remainder Theorem. Thus the number of nonzero summands in~\eqref{eq:vertices} is $O(N/(q_1q_2\cdots q_{k-1}) + 1)$, and it follows that
\[ |v_{r_1,\dots,r_{k-1};f}(a)| \lesssim 1 + \frac{q_1q_2\cdots q_{k-1}}{N} \lesssim 1, \]
as desired. This completes the proof.
\end{proof}

This completes the proof of Theorem~\ref{thm:gw-bound}, which bounds $\GW(T_k/N)$, other than the matching lower bound $\GW(T_3/N) \gtrsim \sqrt{N\log N}$ and the improvement $\GW(T_4/N) \lesssim N^{3/4}(\log N)^{1/4}$, whose proofs can be found in Appendix~\ref{app:gw}.

\section{Variational problem at the macroscopic scale}
\label{sec:macro}

The goal of this section is to prove Proposition~\ref{ppn:rate-extremal}, which reduces the entropic variational problem in the macroscopic scale, i.e., $\delta^{-3}p^{k-2}(\log(1/p))^2 \rightarrow 0$, to a corresponding extremal problem in additive combinatorics:
\begin{equation}
 \label{eq:delta-large}
	\phi_p^{(k,\Omega)}(\delta)
	=  (1+o(1)) \log(1/p) \cdot \min_{S \subset \Omega} \Big\{ |S|: T_k(S) \ge \delta p^k T_k(\Omega)\Big\}
\end{equation}
(provided that $\delta = O(1)$ and $\delta p^kN^2 \to \infty$). Let us provide an overview of the proof strategy.
The upper bound on $\phi_p^{(k,\Omega)}(\delta)$ follows by considering a function which takes value 1 on an interval and $p$ elsewhere (Section \ref{sec:pfupper}). For the lower bound, suppose $T_k(f) \ge (1+\delta) p^k T_k(\Omega)$. Write $f(a) = p + g(a)$. Let $g'$ denote a function obtained from $g$ by changing each $g(a)$ to $0$ if $g(a)$ is already sufficiently close to zero (the exact threshold will be specified in the proof). It will be shown that $T_k(f) \approx p^k T_k(\Omega) + T_k(g')$, so that $T_k(g') \ge (1-o(1)) \delta p^k T_k(\Omega)$. For now, it is fine to pretend that $g'$ is an indicator function of a set, so that we have a lower bound on the number of $k$-APs of the set. We will prove an extremal result on maximizing the number of $k$-APs of a set of given size. This will imply a lower bound on $\sum_{a \in \Omega} g'(a)$, thereby giving the desired lower bound on $\sum_a I_p(f(a)) \approx \sum_a g'(a) \log(1/p)$. 

\subsection{Proof of upper bound in Proposition~\ref{ppn:rate-extremal}}\label{sec:pfupper} We begin by noting that the right-hand side expression of \eqref{eq:delta-large} does not change if $\delta$ is replaced by some $\delta' = \delta + o(\delta)$. This is equivalent to the fact that the maximum number of $k$-APs in a subset of $\Omega$ of size $n$ does not change significantly if $n$ is changed to $n+o(n)$. This is clear from the exact formula in Theorem~\ref{thm:maxAP} and \eqref{eq:Tk(n)} whenever the hypothesis of the theorem applies, or otherwise in general from the easy lemma below (applied with $n \to \infty$, $n \le N$, and $s = o(n)$, so that $M_n \gtrsim n^2$ and $M_{n + o(n)} \sim M_n$).

\begin{lemma}\label{lem:max-kap-stability}
	Let $k\ge 3$ and $\Omega = \Z/N\Z$ or $[N]$. Let $M_n = \max_{A\subset \Omega: |A| \le n} T_k(A)$. Then
	\[
	M_n \le M_{n+s} \le M_n + ks(n+s)
	\] for all $n,s \ge 0$.
\end{lemma}

\begin{proof}
	It is easy to see that $T_k(A \cup \{a\}) \le T_k(A) + k|A| + 1$ by counting the number of new $k$-APs that are formed with the addition of a new element $a$ to $A$. Thus $M_{n+1} \le M_n + kn + 1$. The lemma follows by iterating this bound.
\end{proof}

It is not too hard to prove that $\phi_p^{(k,\Omega)}(\delta)$ is at most the right-hand side quantity in \eqref{eq:delta-large}. Take any $S \subset \Omega$ with $T_k(S) \ge (1-p^k)^{-1}\delta p^kT_k(\Omega)$ (here we are implicitly changing the $\delta$ in the right-hand side of \eqref{eq:delta-large} to $\delta'= (1-p^k)^{-1}\delta=\delta+o(\delta)$), and let $f$ in the variational problem \eqref{eq:var} be the function
\begin{align}\label{eq:interval}
f(a) = \begin{cases} 1 & \text{if } a \in S, \\
                     p & \text{if } a \notin S.
       \end{cases}
\end{align}
So that
\[
T_k(f) \ge (1-p^k)T_k(S) + p^k T_k(\Omega) \ge (1+\delta) p^k T_k(\Omega),
\]
and
\[
\sum_{a \in \Omega} I_p(f(a)) = |S| \log(1/p).
\]
This proves the upper bound in Proposition~\ref{ppn:rate-extremal}.

\subsection{Proof of lower bound in Proposition~\ref{ppn:rate-extremal}} To begin with note that taking $S \subset \Omega$ to be an interval of size $\ceil{\sqrt{k\delta'  T_k(\Omega)}  p^{k/2}}$, where $\delta'= (1-p^k)^{-1}\delta=\delta+o(\delta)$, ensures $T_k(S) \ge (1-p^k)^{-1}\delta p^kT_k(\Omega)$. Then taking $f$ as in \eqref{eq:interval} gives $T_k(f) \ge (1+\delta) p^k T_k(\Omega)$ and $\sum_{a\in \Omega}I_{p}(f(a))=\sqrt{k\delta'  T_k(\Omega)} p^{k/2} \log(1/p)
\leq 2 \sqrt{k\delta} N p^{k/2} \log(1/p)$, for $N$ large enough. Therefore, to show that the left-hand side of \eqref{eq:rateext} is greater than its right-hand side, it suffices to restrict our attention to functions $f \colon \Omega \to [p,1]$ that satisfy
\begin{equation}\label{eq:macro-f-assumption}
\sum_{a\in \Omega}I_{p}(f(a)) \leq 2\sqrt{k\delta} N p^{k/2} \log(1/p) \quad \text{and} \quad T_k(f)\geq (1+\delta)p^kT_k(\Omega)
\end{equation}
(note that we can restrict the range of $f$ to $[p,1]$ since $I_p(\cdot)$ is decreasing in $[0,p]$). We will show that for such $f$,
\[
\sum_{a \in \Omega}I_p(f(a)) \ge
(1+o(1)) \log(1/p) \cdot \min_{S \subset \Omega} \Big\{ |S|: T_k(S) \ge \delta p^k T_k(\Omega)\Big\}.
\]

We recall a useful asymptotic estimate of $I_{p}(\cdot)$ from \cite[Lemma 3.3]{LZ-sparse}.

\begin{lemma}
  \label{lem:entropy-est}
Let $p \to 0$, and $x = x(p) \in [0,1-p]$. If $x=o(p)$, then $I_p(p + x) \sim x^2/(2p)$.
If $x/p \to \infty$, then
$I_p(p + x) \sim x \log(x/p).$
\end{lemma}

Let $f(a) = p + g(a)$, where $g:\Omega\rightarrow [0, 1-p]$. (When $\Omega = [N]$, we set both $f$ and $g$ to be zero outside $[N]$.) Note the following bounds on $g$: by convexity of $I_p(\cdot)$, \eqref{eq:macro-f-assumption}, and Lemma~\ref{lem:entropy-est},
\begin{align*}
I_p\left(p+ \frac{1}{N}\sum_{a\in \Omega}g(a)\right)
&=
I_p\left(\frac{1}{N}\sum_{a\in \Omega}f(a)\right)
\le
\frac{1}{N}\sum_{a\in \Omega}I_{p}(f(a))
\\&\le
\sqrt{k \dd}p^{k/2}\log(1/p)
\sim
I_{p}\left(p+ 2 (k\dd)^{1/4}p^{(k+2)/4}\sqrt{\log(1/p)}\right),
\end{align*}
since $\dd^{1/4}p^{(k+2)/4}\sqrt{\log(1/p)}=o(p)$ whenever $\delta=O(1)$. Moreover,  since $I_{p}(p+x)$ is increasing for $x \in [0,p-1]$,
\begin{equation}\label{eq:macro-g-sum-bound}
\frac{1}{N}\sum_{a\in \Omega}^N g(a) \lesssim {\dd}^{1/4} p^{(k+2)/4}\sqrt{\log(1/p)}.
\end{equation}

The proof of the lower bound on $\sum_{a \in \Omega} I_p(f(a))$ proceeds via two-step thresholding on the function $g$. At each step, we choose some threshold $\tau$ and decompose $g$ into its small and large components:
\[
g = g_{\le \tau} + g_{> \tau}.
\]
Here $g_{\le \tau}(a):=g (a) \bm 1\{g(a)\le \tau\}$, and $g_{> \tau} :=g (a) \bm 1\{g(a) > \tau\}$.
\begin{itemize}
\item[(1)] ({\it Thresholding}) First, we perform the decomposition with $\tau = p^{3/4}$ and show that the contribution to $T_k(f)$ from the small component $g_{\le \tau}$ is negligible.

\item[(2)] ({\it Bootstrapping}) Next, we bootstrap the argument in the first step and take a higher threshold $\tau = p^{o(1)}$.
\end{itemize}

The following lemma will be useful later.

\begin{lemma}\label{lem:double-sum}
Let $\Omega = \Z/N\Z$ or $[N]$. Let $x,y \in \{0, 1, \dots, k-1\}$ with $x \ne y$. For any $f \colon \Omega \to [0,1]$, one has
\[
\sum_{a, b} f({a+ x b})f(a+y b) \leq  (k-1)\biggl(\sum_{a\in \Omega} f(a)\biggr)^2,
\]
where the sum on the left-hand side is taken over all pairs of elements $(a,b)$ in the ambient group such that $\{a+xb,a+yb\} \subset \Omega$.
\end{lemma}

\begin{proof}
The lemma follows from observing that after expanding the right-hand side, for any $c,d \in \Omega$, there are at most $k-1$ pairs of elements $a,b$ in the ambient group such that $a+xb = c$ and $a+yb = d$. Indeed, subtracting the two equations gives $(x-y)b = c-d$. Since $0 < |x-y| \le k-1$, the number of solutions for $b$ is at most 1 when $\Omega = [N]$ and at most $(k-1)$ when $\Omega = \Z/N\Z$.
\end{proof}

\subsubsection{The thresholding step} \label{sec:macro-threshold} In this section we formalize step (1) above. 
Recall that $f$ is as in \eqref{eq:macro-f-assumption}, and $f=p+g$, where $g:\Omega\rightarrow [0, 1]$ satisfies $0\leq g\leq 1-p$. From \eqref{eq:macro-f-assumption} we know that
\[
T_k(p + g_{\le \tau} + g_{> \tau}) = T_k(f) \ge (1+\delta)p^k T_k(\Omega).
\]
The expression $T_k(p + g_{\le \tau} + g_{> \tau})$, when written out as a sum, expands into a number of components. The following lemma shows that, with an appropriate choice of $\tau$, the only non-negligible contributions are $T_k(p) = p^kT_k(\Omega)$ and $T_k(g_{>\tau})$.

\begin{lemma}\label{lem:threshold} Assume that $\delta^{-3}p^{k-2}\log^{2}(1/p) \rightarrow 0$, and $f = p + g\colon \Omega \rightarrow [p, 1]$ satisfies \eqref{eq:macro-f-assumption}. Let $\tau = p^{3/4}$. Then
\begin{equation} \label{eq:threshold-bound}
T_{k}(g_{> \tau}) \ge (1-o(1)) \delta p^{k} T_k(\Omega)
\end{equation}
and
\begin{equation} \label{eq:threshold-big-sum-upper-bound}
\sum_{a\in \Omega}g_{> \tau}(a) \lesssim \sqrt \delta Np^{k/2}.
\end{equation}
\end{lemma}

\begin{remark} In the above lemma, one may take $\tau = p^s$ for any fixed $2/3 < s < 1$.
\end{remark}

\begin{proof} From Lemma~\ref{lem:entropy-est} we have $I_p(p+g_{>\tau}(x)) \asymp g_{>\tau}(x) \log(1/p)$.
Thus
\[
\sum_{a\in \Omega} g_{> \tau}(a) \log(1/p)
\asymp
\sum_{a\in \Omega}I_p(p+g_{> \tau}(a))
\le \sum_{a \in \Omega} I_p(f(a))
\lesssim \sqrt{\delta} N p^{k/2} \log(1/p),
\]
where the final step uses \eqref{eq:macro-f-assumption}. This gives us \eqref{eq:threshold-big-sum-upper-bound}.

By expanding, we have
\begin{align}\label{eq:macro-T-expand}
T_k(p+g_{\le \tau}+g_{> \tau}) = \sum_{X, Y, Z} T_{X,Y,Z}(p,g_{\le \tau},g_{> \tau}),
\end{align}
where the sum is over all ordered partitions $(X, Y, Z)$ of the set $\{0, 1, \dots, k-1\}$, and
\begin{equation}
	\label{eq:T_XYZ}
T_{X,Y,Z}(p,g_{\le \tau},g_{> \tau}) := \sum_{a,b} p^{|X|} \prod_{y \in Y} g_{\le \tau}(a + yb) \prod_{z \in Z} g_{> \tau}(a + zb).
\end{equation}
Here the sum is taken over all pairs of elements $(a,b)$ in the ambient group such that $\{a, a+b, \dots, a+(k-1)b\} \subset \Omega$.
We say that $T_{X,Y,Z}(p,g_{\le \tau},g_{> \tau})$ {\it contributes negligibly} to the sum \eqref{kap} (or \emph{negligible} for short) if $T_{X,Y,Z}(p,g_{\le \tau},g_{> \tau})=o(\delta N^2p^k)$.  We will show that all terms except for $(X,Y,Z) = (\{0, \dots, k-1\}, \emptyset,\emptyset)$ and $(\emptyset, \emptyset, \{0, \dots, k-1\})$ are negligible, i.e., the only non-negligible terms are $p^kT_k(\Omega)$ and $T_k(g_{>\tau})$. This would prove \eqref{eq:threshold-bound}, due to the assumption $T_k(f) \ge (1+\delta)p^k T_k(\Omega)$ from \eqref{eq:macro-f-assumption}.

First, if $|Z| \ge 2$, then by Lemma~\ref{lem:double-sum} and \eqref{eq:threshold-big-sum-upper-bound},
\[
T_{X,Y,Z}(p,g_{\le \tau},g_{> \tau})
\le (k-1) \tau^{|X \cup Y|} \left(\sum_{a\in \Omega} g_{> \tau}(a)\right)^2
\lesssim \tau^{|X \cup Y|} \delta N^2 p^k.
\]
Therefore, if $|Z| \ge 2$, then the contribution from $T_{X,Y,Z}(p,g_{\le \tau},g_{> \tau})$ is negligible unless $|X \cup Y| = 0$, i.e., $(X,Y,Z) = (\emptyset, \emptyset, \{0, \dots, k-1\})$.

Next, if $|X| = k-1$ and $|Y \cup Z| = 1$, then by \eqref{eq:macro-g-sum-bound},
\[
T_{X,Y,Z}(p,g_{\le \tau},g_{> \tau})
\le k N p^{k-1} \sum_{a \in \Omega} g(a)
\lesssim
\delta^{1/4} N^2 p^{(5k-2)/4}\sqrt{\log(1/p)}
=o(\delta N^2p^k),
\]
where in the final step we use the macroscopic scale assumption $\delta^{-3}p^{k-2}(\log(1/p))^2 \rightarrow 0$. Therefore these terms are negligible.

Finally, if $|Z| \le 1$ and $|Y\cup Z|\geq 2$, then  by \eqref{eq:macro-g-sum-bound} and Lemma~\ref{lem:double-sum},
\[
T_{X,Y,Z}(p,g_{\le \tau},g_{> \tau})
\le (k-1) \tau^{k-2} \left(\sum_{a \in \Omega} g(a) \right)^2
\lesssim
\tau^{k-2} N^2\delta^{1/2} p^{(k+2)/2} \log(1/p)
=o(\delta N^2p^k),
\]
where the last step holds due to $\tau = o(p^{2/3}(\log(1/p))^{-2/(3k-6)})$ and the macroscopic scale assumption on $\delta$.

It follows from the above analysis that the only non-negligible contributions to the sum \eqref{eq:macro-T-expand} are  $(X,Y,Z) = (\{0, \dots, k-1\}, \emptyset,\emptyset)$ and $(\emptyset, \emptyset, \{0, \dots, k-1\})$, so that
\[
T_k(f) = p^kT_k(\Omega) + T_k(g_{> \tau}) + o(\delta p^k N^2).
\]
By the assumption $T_k(f) \ge (1+\delta)p^k T_k(\Omega)$ from \eqref{eq:macro-f-assumption}, we obtain $ T_k(g_{> \tau}) \ge (1-o(1))\delta p^k T_k(\Omega)$ as desired.
\end{proof}

\subsubsection{The bootstrapping step}
Now, we strengthen Lemma~\ref{lem:threshold} by replacing $\tau$ with any $\tau = o(1)$.

\begin{lemma}\label{lem:bootstrap}
Assume that $\delta^{-3}p^{k-2}\log^{2}(1/p) \rightarrow 0$, and $f = p + g\colon \Omega \rightarrow [p, 1]$ satisfies \eqref{eq:macro-f-assumption}. For any $\tau = o(1)$, we have
\[
T_{k}(g_{> \tau}) \ge (1-o(1)) \delta p^{k} T_k(\Omega).
\]
\end{lemma}

\begin{proof}
Observe that
\begin{equation}\label{eq:bootstrap-split-bound}
T_k(g_{> p^{3/4}}) - T_k(g_{> \tau}) \le k^2 \tau \left( \sum_{a \in \Omega} g_{>p^{3/4}}(a)\right)^2 \lesssim \tau \delta p^k N^2=  o(\delta p^k N^2).
\end{equation}
To prove the first inequality, note that the left-hand side can be bounded above by a sum of $k$ terms, where the $j$-th term (for $1 \leq j \leq k$) is itself the following sum 
\[
\sum_{a,b \in \Omega} g_1(a)g_2(a+b) \cdots g_k(a+(k-1)b)
\]
where $g_j = g_{\le \tau}$ and all other $g_i$'s are set to $g_{> p^{3/4}}$, and this sum can be bounded by $\tau (k-1)(\sum_{a \in \Omega} g_{>p^{3/4}})^2$ using Lemma~\ref{lem:double-sum}. The second part of \eqref{eq:bootstrap-split-bound} follows from \eqref{eq:threshold-big-sum-upper-bound}. By \eqref{eq:threshold-bound} we have $T_k(g_{> p^{3/4}}) \ge (1-o(1))\delta p^k T_k(\Omega)$, and thus by \eqref{eq:bootstrap-split-bound} we have $T_k(g_{> \tau}) \ge (1-o(1))\delta p^k T_k(\Omega)$ as well.
\end{proof}

\subsubsection{Completing the proof of Proposition~\ref{ppn:rate-extremal}}

Assume that $f \colon \Omega \to [p,1]$ satisfies \eqref{eq:macro-f-assumption}. Let $f = p + g$, and choose some threshold $\tau$ satisfying $\tau = p^{o(1)} \to 0$ (e.g., $\tau = 1/\log(1/p)$ for concreteness). By Lemma~\ref{lem:bootstrap}, one has
\[
T_k(g_{> \tau}) \ge (1 - o(1)) \delta p^k T_k(\Omega).
\]
This implies, by Lemma~\ref{lem:weighted-AP-count} below (used in the contrapositive in conjunction with Lemma~\ref{lem:max-kap-stability}),
\[
\sum_{a \in \Omega} g_{> \tau} (a) \ge
(1+o(1)) \min_{S \subset \Omega} \Big\{ |S|: T_k(S) \ge \delta p^k T_k(\Omega)\Big\}.
\]
By Lemma~\ref{lem:entropy-est}, $I_p(p+g_{>\tau}(a)) \sim g_{>\tau}(x)\log(1/p)$ since $\tau = p^{o(1)}$. Therefore,
\begin{align*}
\sum_{a \in \Omega} I_p(f(a)) \ge \sum_{a \in \Omega} I_p(p + g_{>\tau}(a))
&= (1+o(1)) \log(1/p)\sum_{a \in \Omega} g_{> \tau} (a)
\\
&\ge (1+o(1)) \log(1/p) \min_{S \subset \Omega} \Big\{ |S|: T_k(S) \ge \delta p^k T_k(\Omega)\Big\},
\end{align*}
thereby completing the proof of Proposition~\ref{ppn:rate-extremal}, modulo the following lemma, which says  that the problem of maximizing the number of $k$-APs in a set remains roughly unchanged even if we allow the elements to be weighted.

\begin{lemma}
\label{lem:weighted-AP-count}
Let $f \colon \Omega \to [0,1]$ be such that $\sum_{a\in \Omega}f(a) = m $. Then
\begin{equation} \label{eq:weighted-GS}
T_k (f) \le (1+o(1))  \max_{A\subset \Omega: |A| \le m} T_k(A)
 \end{equation}
provided that $m \to \infty$ as $N \to \infty$.
\end{lemma}

\begin{proof}
Let $M_s = \max_{A\subset \Omega: |A| \le s} T_k(A)$.
Let $\Omega_f$ be a random subset of $\Omega$ chosen by including element $a\in \Omega$ with probability $f(a)$ independently for all $a \in \Omega$. Note that $\E[|\Omega_f|] = m$, and for $a, b \in \Omega$,
\[
f(a)f(a+b)\cdots f(a+(k-1)b) \le \P(a, a+b,\dots a+(k-1)b \in A).
\]
(it is always an equality when the elements of the $k$-AP are distinct). This implies
\[
T_k(f) \leq \E T_k(\Omega_f) \leq \sum_{s \ge 1} \P(|\Omega_f| = s) M_s.
\]
For $s \le m + m^{2/3}$, we bound $M_s \le M_{m+m^{2/3}} = (1+o(1))M_m$.
And for $s > m + m^{2/3}$ we use the trivial bound $M_s \le s^2$. Thus
\begin{align*}
T_k(f) &\le (1+o(1)) M_m + 4m^2 \P(m + m^{2/3} < |A| \le 2m)  + \sum_{s > 2m} \P(|A| = s)s^2
\\
&\le (1+o(1))M_m + 4m^2 \P(|A| > m + m^{2/3})  + \sum_{s > 2m} \P(|A| \ge s)s^2
\\
&\le (1+ o(1)) M_m + 4m^2e^{-m^{1/3}/3} + \sum_{s > 2m} s^2e^{-s/6}
\\
&= (1+o(1))M_m,
\end{align*}
where the penultimate step uses Chernoff bound in the following form: if $X$ is a sum of independent indicator random variables, and $\E X = \mu$, then for any $\delta > 0$, $\P(X \ge (1+\delta)\mu) \le e^{-\min\{\delta^2,\delta\}\mu/3}$.
\end{proof}

This completes the proof of Proposition~\ref{ppn:rate-extremal}. Combining it with Theorem~\ref{thm:maxAP}, which will be proved in Section~\ref{sec:kAPmax}, yields Theorem~\ref{thm:rate}, the asymptotic solution to the variational problem in the macroscopic scale.

\section{Variational problem at the microscopic scale}
\label{sec:micro}

In this section, we prove Theorem~\ref{thm:micro-rate}, which solves the variational problem in the microscopic scale, i.e., $\delta^{-3}p^{k-2}\log^{2}(1/p) \rightarrow \infty$.  The following theorem unifies the settings $\Omega = [N]$ and $\Z/N\Z$.
\begin{theorem} \label{thm:micro-rate-general}
Fix $k\ge 3$. Let $\Omega =[N]$ or $\Z / N\Z$ . Then for $p = p_N \rightarrow 0$ and $\delta = \delta_N > 0$ such that $\delta^{-3} p^{k-2}(\log(1/p))^2 \to \infty$, we have
\begin{equation}\label{eq:micro-rate-general}
\phi_p^{(k, \Omega)}(\delta)=(1+o(1))\frac{\dd^2 T_k(\Omega)^2 p}{2\sum_{a\in \Omega}\nu_a^2},
\end{equation}
where $\nu_a$ is the number of $k$-APs in $\Omega$ containing $a \in \Omega$ (the constant $k$-AP $a,a,\dots,a$ is counted $k$ times), i.e., $\nu_a$ is the number of triples $(x,y,j)$ where $x,y$ are elements in the ambient group, $j \in \{0, 1, \dots, k-1\}$, so that $a = x + jy$ and $\{x,x+y, \dots, x+(k-1)y\} \subset \Omega$.
\end{theorem}

In Section~\ref{sec:micro-proof}, we prove Theorem~\ref{thm:micro-rate-general}, following similar thresholding strategy to the macroscopic setting. In Section~\ref{sec:micro-rate}, we then compute the rate formulae in Theorem~\ref{thm:micro-rate}.

\subsection{Proof of Theorem \ref{thm:micro-rate-general}} \label{sec:micro-proof}
We begin by showing that $\phi_p^{(k, \Omega)}(\delta)$ is at most the right-hand side of \eqref{eq:micro-rate-general}. The claim follows from an explicit construction of a function $f$ in the variational problem \eqref{eq:var}, as given by the following lemma.

\begin{lemma}\label{lm:microexample}
Let $\Omega = [N]$ or $\ZZ/N\ZZ$, and let $\nu_a$ be defined as in Theorem~\ref{thm:micro-rate-general}. Define $g \colon \Omega \to [0,1]$ by
\[
g(a) = \frac{\delta T_k(\Omega) \nu_a}{\sum_{a\in \Omega}\nu_a^2}p.
\]
Then $f = p + g$ satisfies
\begin{align}\label{eq:microexample}
I_p(f(a)) \sim \frac{\dd^2 T_k(\Omega)^2 p}{2\sum_{a\in \Omega}\nu_a^2} \quad \text{and} \quad 
T_k(f) \geq (1+\delta)p^k T_k(\Omega).
\end{align}
As a consequence, $\phi_p^{(k, \Omega)}(\delta) \leq (1+o(1))\frac{\dd^2 T_k(\Omega)^2 p}{2\sum_{a\in \Omega}\nu_a^2}$. 
\end{lemma}

\begin{proof}
We have $\nu_a \asymp N$, $\sum_{a \in \Omega} \nu_a^2 \asymp N^3$, so $g(a) \asymp \delta p = o(p)$, and thus by Lemma~\ref{lem:entropy-est},
\[
\sum_{a \in \Omega} I_p(f(a))
\sim \sum_{a \in \Omega} \frac{g(a)^2}{2p}
\sim \frac{\dd^2 T_k(\Omega)^2 p}{2\sum_{a\in \Omega}\nu_a^2}.
\]
Next, by expanding we have
\begin{align*}
T_k(f) = T_k(p+g)
&\ge p^k T_k(\Omega) + p^{k-1} \sum_{x,y} \sum_{j=0}^{k-1} g(x+jy)
\\
&= p^k T_k(\Omega) + p^{k-1} \sum_{a \in \Omega} g(a)\nu_a = (1+\delta)p^k T_k(\Omega),
\end{align*}
where $(x,y)$ ranges over all pairs of elements in the ambient group such that $\{x, x+y, \dots, x+(k-1)y\} \subset \Omega$. This proves \eqref{eq:microexample} and the upper bound on $\phi_p^{(k, \Omega)}(\delta)$ in Theorem~\ref{thm:micro-rate-general}.
\end{proof}

Now we prove the lower bound on $\phi_p^{(k, \Omega)}(\delta)$. To begin with, using $\sum_{a \in \Omega} \nu_a = kT_k(\Omega)$ and the Cauchy--Schwarz inequality, we have
\begin{equation} \label{eq:micro-equal}
\phi_p^{(k, \Omega)}(\delta)
\le
(1+o(1)) \frac{\dd^2 T_k(\Omega)^2 p}{2\sum_{a\in \Omega}\nu_a^2}
\le (1+o(1)) \frac{\delta^2 p N}{2k^2}.
\end{equation}
Therefore, to prove the lower bound on $\phi_p^{(k, \Omega)}(\delta)$ in \eqref{eq:micro-rate-general}, we can restrict our attention to functions $f = p + g \colon \Omega \to [p,1]$ satisfying
\begin{equation}
   \label{eq:micro-f-assumption}
\sum_{a\in \Omega} I_p(f(a)) \le \delta^2 N p
\quad \text{and}
\quad
 T_k(f)\geq (1+\delta)p^k T_k(\Omega).
\end{equation}
Then by convexity, \eqref{eq:micro-f-assumption}, and Lemma \ref{lem:entropy-est},
\[
I_p\left(p+ \frac{1}{N}\sum_{a\in \Omega}g(a)\right)
\le
\frac{1}{N}\sum_{a\in \Omega} I_p(p+g(a))
\le \delta^2 p
\sim
I_p(p+\sqrt{2} \delta p).
\]
Since $I_p(p+x)$ is increasing for $x \in [0,1-p]$,  we have
\begin{equation}\label{eq:micro-g-sum-bound}
\sum_{a \in \Omega} g(a) \lesssim \delta p N.
\end{equation}

The following lemma gives a lower bound on the weighted average of any function $g$ satisfying  $f = p + g\colon \Omega \rightarrow [p, 1]$ with $T_k(f) \ge (1+\delta)p^k T_k(\Omega)$, in the microscopic regime. 

\begin{lemma}\label{lem:threshold} Assume that $\delta^{-3}p^{k-2}\log^{2}(1/p) \rightarrow \infty$, and $f = p + g\colon \Omega \rightarrow [p, 1]$ with $T_k(f) \ge (1+\delta)p^k T_k(\Omega)$. Then \begin{equation} \label{eq:micro-g-sum-lower-bd}
\sum_{a \in \Omega} g(a)\nu_a \ge (1 - o(1))\delta p T_k(\Omega).
\end{equation}
\end{lemma}

\begin{proof} Set threshold $\tau = p^{3/4}$.
As in Section~\ref{sec:macro-threshold}, write $g = g_{\le \tau} + g_{> \tau}$.
As in Lemma~\ref{lem:threshold}, we have by Lemma~\ref{lem:entropy-est} and \eqref{eq:micro-f-assumption},
\begin{equation}\label{eq:micro-g-large-sum-bound}
\sum_{a\in \Omega} g_{> \tau}(a) \log(1/p)
\asymp
\sum_{a\in \Omega}I_p(p+g_{> \tau}(a))
\le \sum_{a \in \Omega} I_p(f(a))
\le \delta^2 N p,
\end{equation}
thereby gaining an extra $\log(1/p)$ factor compared to \eqref{eq:micro-g-sum-bound}.  

By expanding, we have
\begin{equation}
	\label{eq:micro-T-expand}
	T_k(p+g_{\le \tau}+g_{> \tau}) = \sum_{X, Y, Z} T_{X,Y,Z}(p,g_{\le \tau},g_{> \tau}),
\end{equation}
where $T_{X,Y,Z}(p,g_{\le \tau},g_{> \tau})$ is the same as earlier \eqref{eq:T_XYZ}.
We say that $T_{X,Y,Z}(p,g_{\le \tau},g_{> \tau})$ \emph{contributes negligibly} to the sum \eqref{eq:micro-T-expand} if $T_{X,Y,Z}(p,g_{\le \tau},g_{> \tau}) = o(\delta N^2 p^k)$. We will show that unless $|X| = k-1$ or $k$, the term contributes negligibly.

Indeed, as in Section~\ref{sec:macro-threshold}, if $|Z| \ge 2$, then by Lemma~\ref{lem:double-sum} and \eqref{eq:micro-g-large-sum-bound},
\[
T_{X,Y,Z}(p,g_{\le \tau},g_{> \tau})
\le (k-1) \left(\sum_{a\in \Omega} g_{> \tau}(a)\right)^2
\lesssim \left( \frac{\delta^2 Np}{\log(1/p)}\right)^2
= o(\delta N^2p^k),
\]
where the final step uses the microscopic scale hypothesis $\delta^{-3}p^{k-2}\log^{2}(1/p) \rightarrow \infty$.

If $|Z| \le 1$ and $|Y\cup Z|\geq 2$, then  by Lemma~\ref{lem:double-sum} and \eqref{eq:micro-g-sum-bound},
\[
T_{X,Y,Z}(p,g_{\le \tau},g_{> \tau})
\le (k-1) \tau^{k-2} \left(\sum_{a \in \Omega} g(a) \right)^2
\lesssim
\tau^{k-2} \delta^2 p^2 N^2
=o(\delta N^2p^k),
\]
where the last step holds due to $\tau = o(p^{2/3}(\log(1/p))^{-2/(3k-6)})$ and the microscopic scale hypothesis on $\delta$.

This shows that the non-neglible contributions are those terms with $|X| = k-1$ or $k$. Hence
\begin{align*}
T_k(f) &= p^k T_k(\Omega) + p^{k-1} \sum_{x,y} \sum_{j=0}^{k-1} g(x+jy) + o(\delta p^k N^2)
\\
&= p^k T_k(\Omega) + p^{k-1} \sum_{a \in \Omega} g(a)\nu_a  + o(\delta p^k N^2),
\end{align*}
where the first sum runs over all pairs $(x,y)$ of elements in the ambient group such that $\{x,x+y,\dots,x+(k-1)y\} \subset \Omega$.
Since $T_k(f) \ge (1+\delta)p^k T_k(\Omega)$, we have $\sum_{a \in \Omega} g(a)\nu_a \ge (1 - o(1))\delta p T_k(\Omega)$, as required. 
\end{proof}

In the final step of the proof of Theorem~\ref{thm:micro-rate-general}, we convert the lower bound on the weighted sum of $g$ from the above lemma to a lower bound on the entropy $\sum_a I_p(p+g(a))$. We consider the two cases $\Omega = \Z/N\Z$ and $[N]$ separately.

\medskip

\emph{Case 1:} $\Omega = \Z/N\Z$. In this case, we have $\nu_a = kN$ for all $a \in \Omega$ and $T_k(\Omega) = N^2$. By convexity of $I_p(\cdot)$, \eqref{eq:micro-g-sum-lower-bd}, and Lemma~\ref{lem:entropy-est},
\begin{align*}
\sum_{a \in \Omega} I_p(p+g(a))
&\ge N I_p\left(p + \frac{1}{N} \sum_{a \in \Omega} g(a)\right) 
\\
&\ge
N I_p\left(p + (1-o(1))\frac{\delta p}{k}\right) 
\sim  \frac{\delta^2 Np}{2k^2}.
\end{align*}
This combined with Lemma \ref{lm:microexample}, proves Theorem~\ref{thm:micro-rate-general},  when $\Omega = \Z/N\Z$.

\medskip

\emph{Case 2:} $\Omega = [N]$. In this case, the quantities $\nu_a$ are unequal, and the solution requires an extra step. We use the estimate\footnote{\emph{Proof of estimate}: by Lemma~\ref{lem:entropy-est}, for $x \in [p^{1.1},\tfrac12 p^{0.9}]$ we have $I_p(p+x) \ge I_p(p+p^{1.1}) \sim \tfrac12 p^{1.2} \ge p^{0.3} x$, and for $x \ge \tfrac12 p^{0.9}$ we have $I_p(p+x) \sim x \log(x/p) \ge (1+o(1)) p^{0.3} x$.}
\begin{equation}\label{eq:ent-lin-est}
I_p(p+x) \ge (1+o(1)) p^{0.3} x \quad \text{uniformly for all } x \in [p^{1.1},1-p].
\end{equation}
It follows by \eqref{eq:ent-lin-est} and \eqref{eq:micro-f-assumption} that
\begin{align*}
\sum_{a \in \Omega} g_{>p^{1.1}}(a)
&\le (1+o(1)) p^{-0.3} \sum_{a \in \Omega} I_p(p+g_{>p^{1.1}}(a))
\\
&\le (1+o(1)) p^{-0.3} \sum_{a \in \Omega}I_p(f(a))
\le (1+o(1))\delta^2 p^{0.7} N
= o(\delta p N)
\end{align*}
where the final step, $\delta = o(p^{0.3})$, is due to the microscopic scale hypothesis. Since $\nu_a \asymp N$ for all $a \in \Omega$, the above estimate along with \eqref{eq:micro-g-sum-lower-bd} gives
\[
\sum_{a \in \Omega} g_{\le p^{1.1}} (a)\nu_a \ge (1-o(1))\delta pT_k(\Omega).
\]
By the Cauchy--Schwarz inequality, we have
\[
\sum_{a \in \Omega} g_{\le p^{1.1}} (a)^2 \ge (1-o(1))\frac{\delta^2 p^2T_k(\Omega)^2}{\sum_{a \in \Omega} \nu_a^2}.
\]
It follows by Lemma~\ref{lem:entropy-est} and the above estimate that
\[
\sum_{a \in \Omega} I_p(f(a))
\ge \sum_{a \in \Omega} I_p(p + g_{\le p^{1.1}}(a))
\sim \sum_{a \in \Omega} \frac{g_{\le p^{1.1}}(a)^2}{2p}
\ge (1+o(1)) \frac{\delta^2 T_k(\Omega)^2}{2\sum_{a \in \Omega} \nu_a^2},
\]
which  combined with Lemma \ref{lm:microexample} completes the proof of Theorem~\ref{thm:micro-rate-general}.

\subsection{Microscopic rate function} \label{sec:micro-rate}

In the case $\Omega=\Z/N\Z$, by symmetry, $\nu_s = kN$ for all $s \in \Omega$, and $T_k(\Omega) = N^2$. Thus Theorem~\ref{thm:micro-rate} for $\Omega = \Z/N\Z$ follows from Theorem~\ref{thm:micro-rate-general}.

When $\Omega = [N] \subset \Z$, the derivation of the formula in Theorem~\ref{thm:micro-rate} is routine though a bit more involved. For each $s \in [N]$ and $0 \le j < k$, let $\nu_{s,j}$ denote the number of pairs of $a,b \in \Z$ such that $a,a+(k-1)b \in [N]$ and $a + jb = s$, i.e., the number of $k$-APs (allowing zero or negative common difference) contained in $[N]$ and whose $(j+1)$-th term is $s$. It is easy to see that $\nu_{s,j}$ equals to the number of $b \in \Z$ satisfying $1 \le s + (k-1-j)b \le N$ and $1 \le s - jb \le N$, so that
\[
\nu_{s,j} = \min\left\{ \left\lfloor \frac{s-1}{j} \right\rfloor, \left\lfloor \frac{N-s}{k-1-j} \right\rfloor \right\}
+ \min\left\{ \left\lfloor \frac{s-1}{k-1-j} \right\rceil, \left\lceil \frac{N-s}{j} \right\rfloor \right\} + 1.
\]
Thus, for each $s \in [N]$,
\[
\nu_s = \sum_{j=0}^{k-1} \nu_{s,j} = k + 2\sum_{j=0}^{k-1}  \min\left\{ \left\lfloor \frac{s-1}{j} \right\rfloor, \left\lfloor \frac{N-s}{k-1-j} \right\rfloor \right\}.
\]
By Riemann sum, we have
\[
\lim_{N \to \infty} N^{-3} \sum_{s=1}^N \nu_s^2
= 4\sum_{i,j=0}^{k-1} \beta_{i,j},
\]
where
\[
\beta_{i,j} = \int_0^1 \min\left\{ \frac{x}{i}, \frac{1-x}{k-1-i}\right\}
\min\left\{ \frac{x}{j}, \frac{1-x}{k-1-j}\right\} \, dx.
\]
Observing that $\min\{x/i,(1-x)(k-1-i)\}$ is piecewise-linear with the kink at $x = i/(k-1)$, we can compute the above integral: for all $0 \le i \le j \le k-1$,
\begin{align*}
\beta_{ij}
&=
\int_0^{\frac{i}{k-1}} \frac{x^2}{ij} dx
+
\int_{\frac{i}{k-1}}^{\frac{j}{k-1}} \frac{(1-x)x}{(k-1-i)j} dx
+
\int_{\frac{j}{k-1}}^{1} \frac{(1-x)^2}{(k-1-i)(k-1-j)} dx
\\
&=
\frac{(k-1)^2-i^2-(k-1-j)^2}{6 (k-1)^2(k-1-i)j}.
\end{align*}
In particular, for all $0 \le i \le k-1$,
\[
\beta_{i,i} = \frac{1}{3(k-1)^2}.
\]
Therefore,
\[
\lim_{N \to \infty} N^{-3} \sum_{s=1}^N \nu_s^2
= 4\sum_{i,j=0}^{k-1} \beta_{i,j}
= 4 \sum_{i=0}^{k-1} \beta_{i,i}+ 8 \hspace{-.5em}\sum_{0 \le i < j < k} \beta_{i,j}
= \frac{\gamma_k}{(k-1)^2},
\]
where
\[
\gamma_k
= \frac{4}{3}\left(k + \sum_{0\le i < j < k}
\frac{(k-1)^2-i^2-(k-1-j)^2}{(k-1-i)j}\right).
\]
Since $T_k([N]) \sim N^2/(k-1)$ for fixed $k$ as $N \to \infty$, Theorem~\ref{thm:micro-rate} for $\Omega = [N]$ follows from Theorem~\ref{thm:micro-rate-general}.

As for the remark following Theorem~\ref{thm:micro-rate}, we always have
$\gamma_k \ge k^2$
due to \eqref{eq:micro-equal}. The first few values of $\gamma_k$ are
\[
\gamma_3= \frac{28}{3},
\quad
\gamma_4 = 17,
\quad
\gamma_5 = \frac{718}{27}.
\]
The asymptotic dependence of $\gamma_k$ on $k$ can be computed via a Riemann integral\footnote{The integral was computed using \textsc{Mathematica}.} (note that the integrand takes value in $[0,2]$ in the given domain):
\[
\lim_{k \to 0}\frac{\gamma_k}{k^2}
= \frac{4}{3}\int_{0\le x \le y \le 1} \hspace{-1em} \frac{1 - x^2 - (1-y)^2}{(1-x)y} \, dxdy
= \frac{40 - 2\pi^2}{9} \approx 1.14.
\]

\section{Replica symmetry} \label{sec:replica}

In this section, we record a partial result on exact replica symmetry for constant values of $p$ and $\delta$ in the case of $\Omega = \ZZ/N\ZZ$, analogous to results about dense random graphs in \cite{CV11,LZ15}. Unlike previous sections, where we solve the variational problem asymptotically as $p \to 0$, the following theorem gives \emph{exact} replica symmetry, i.e., we give sufficient conditions on constants $p$ and $\delta$ so that the constant function uniquely minimizes the variational problem. Unlike the results in \cite{LZ15}, we do not know if the following theorem gives the full replica symmetry phase (it probably does not). The proof is nearly identical to the one in \cite{LZ15}, the only difference being the H\"older-like inequality (Lemma~\ref{lem:ap-holder}) for $k$-APs.

Using arguments very similar to \cite{Zhao-lower}, one can also prove regions of replica symmetry in the lower tail. Details are omitted.

\begin{theorem}\label{thm:replica}
Take any $0 < p \le q < 1$, positive integer $k\ge 3$ and prime $N$. Suppose that $(q^{k/2}, I_p(q))$ lies on the convex minorant of the function $x \mapsto I_p(x^{2/k})$. Then the constant function $f \equiv q$ is the unique minimizer to the variational problem \eqref{eq:var} with $\Omega = \Z/N\Z$ and $(1 + \delta)p^k = q^k$, so that $\phi_p^{(k,\ZZ/N\ZZ)}(\delta) = N I_p(q)$.
\end{theorem}

\begin{remark}
The hypothesis that $N$ is prime is mainly for convenience, and it is likely unnecessary here. For example, the proof shows that when $k$ is even, there is no requirement on $N$, and when $k$ is odd, $\gcd(k-1,N) = 1$ suffices.
\end{remark}

\begin{proof}
In the variational problem \eqref{eq:var}, suppose that $f \colon \Z/N\Z \to [0,1]$ satisfies $T_k(f) \ge (1+\delta)p^k N^2 = q^k N^2$.
By Lemma~\ref{lem:ap-holder} below, we have
\begin{equation}
	\label{eq:replica-power-ineq}
\sum_{a \in \Z/N\Z} f(a)^{k/2} \ge q^{k/2} N.
\end{equation}
Let $J(x) = I_p(x^{2/k})$, and $\ell(x) = J'(q^{k/2})(x - q^{k-2}) + I_p(q)$. Then $\ell$ is the tangent line to $J(x)$ at $x = q^{k/2}$, and by the convex minorant condition, we have $J(x) \ge \ell(x)$ for all $x \in [0,1]$. Since $I_p(x)$ is increasing in $x \in [p,1]$, we have $J'(q^{k/2}) \ge 0$. It follows that
\begin{align*}
\sum_{a \in \Omega} I_p(f(a))
&= \sum_{a \in \Omega} J(f(a)^{k/2})
\\
&\ge \sum_{a \in \Omega} \ell(f(a)^{k/2})
\\
&= J'(q^{k/2}) \left( \sum_{a \in \Omega} f(a)^{k/2} - q^{k/2} N\right) + I_p(q)
\\
&\ge I_p(q)
\end{align*}
by \eqref{eq:replica-power-ineq}. Equality occurs if and only if $f(a) = q$ for all $a \in \Z/N\Z$.
\end{proof}

\begin{lemma}\label{lem:ap-holder}
Let $k \ge 3$ and $N$ a prime.
For any $f \colon \Z/N\Z \to [0,\infty)$, one has
\[
T_k(f) \le \biggl(\sum_{a \in \Z/N\Z} f(a)^{k/2}\biggr)^2.
\]
\end{lemma}

\begin{proof}
Define
\[
h_{i,j} (a,b) = \sqrt{f(a+ib) f(a+jb)}.
\]
By H\"older's inequality, one has
\begin{align*}
T_k(f)
&= \sum_{a,b \in \Z/N\Z} f(a)f(a+b) \cdots f(a+(k-1)b)
\\
&=
 \sum_{a,b \in \Z/N\Z} h_{0,1}(a,b)h_{1,2}(a,b) \cdots h_{k-2,k-1}(a,b) h_{k-1,0}(a,b)
\\
&\le \prod_{(i,j) \in \{(0,1), (1,2), \dots, (k-2,k-1),(k-1,0)\}} \left(\sum_{a,b \in \Z/N\Z} h_{i,j}(a,b)^k\right)^{1/k}
\\
&= \biggl(\sum_{a \in \Z/N\Z}f(a)^{k/2}\biggr)^2,
\end{align*}
where the last step is due to
\begin{align*}
\sum_{a,b \in \Z/N\Z} h_{ij}(a,b)^k
&= \sum_{a,b \in \Z / N\Z} f(a+ib)^{k/2}f(a+jb)^{k/2}
\\
&= \biggl( \sum_{a\in\Z/N\Z}f(a)^{k/2}\biggr)^2 \qquad \text{if } \gcd(i-j,N) = 1.
\end{align*}
Note that when $k$ is even, we can instead write
\[
f(a)f(a+b)\cdots f(a+(k-1)b)
= h_{0,1}(a,b)^2 h_{2,3}(a,b)^2 \cdots h_{k-2,k-1}(a,b)^2
\]
and remove the need for primality hypothesis on $N$.
\end{proof}

\section{Maximizing the number of $k$-APs}\label{sec:kAPmax}

In this section we prove Theorem~\ref{thm:maxAP}, repeated below for convenience.

\begin{reptheorem}{thm:maxAP}
Fix a positive integer $k \geq 3$. There exists some constant $c_k > 0$ such that the following statement holds. Let $A \subset \Z$ be a subset with $|A| = n$, or $A \subset \Z/N\Z$ with $N$ prime and $|A| = n \leq c_kN$. Then $T_k(A) \leq T_k([n])$.
\end{reptheorem}

In Section~\ref{sec:kAPmax-Z}, we will prove Theorem~\ref{thm:maxAP} when $A \subset \Z$ using a simple combinatorial argument. Unfortunately this proof does not extend to the case $A \subset \Z/N\Z$, due to the lack of a natural ordering in $\Z/N\Z$. Following the idea in~\cite{GS08}, we will attempt to replace the original set $A \subset \Z/N\Z$ by a Freiman model $\widetilde{A} \subset \Z$ (so that in particular $|A| = |\widetilde{A}|$ and $T_k(A) = T_k(\widetilde{A})$). This technique, called \emph{rectification}, was first investigated in~\cite{BLR98}. The following lemma gives a simple example of rectification:

\begin{lemma}\label{lem:simple-rect}
Let $N$ be a positive integer. Let $\widetilde{A} \subset \Z \cap (-N/4, N/4)$ be a subset, and let $A \subset \Z/N\Z$ be the image of $\widetilde{A}$ under the natural projection $\Z \to \Z/N\Z$. Then $T_k(A) = T_k(\widetilde{A})$.
\end{lemma}

\begin{proof}
Let $\pi\colon \widetilde{A} \to A$ be the natural projection map. We need to show that both $\pi$ and $\pi^{-1}$ preserve $k$-APs. It suffices to show that $\pi$ is a \emph{Freiman isomorphism}, in the sense that for any $a_1,a_2,a_3,a_4 \in \widetilde{A}$, $a_1+a_2 = a_3+a_4$ if and only if $\pi(a_1) + \pi(a_2) = \pi(a_3) + \pi(a_4)$. The only if direction is clear. The if direction follows from the fact that
\[  -N < a_1 + a_2 - a_3 - a_4 < N \]
whenever $a_1,a_2,a_3,a_4 \in \widetilde{A}$.
\end{proof}

A more sophisticated rectification lemma is given in~\cite[Theorem 4.1]{GS08}, which allows us to prove Theorem~\ref{thm:maxAP} when the set $A$ has small doubling, in the sense that $|A+A|/|A|$ is small (see Lemma~\ref{lem:kAPmax-small-doubling} below). After stating some preparatory lemmas in Section~\ref{sec:kAPmax-prep}, we will then prove the general case of Theorem~\ref{thm:maxAP} in Section~\ref{sec:kAPmax-Z/N} using a structural decomposition theorem~\cite[Proposition 3.2]{GS08}, which allows us to deduce that if $T_k(A)$ is close to maximal then $A$ must have small doubling.

\subsection{Proof of Theorem~\ref{thm:maxAP} when $A \subset \Z$}\label{sec:kAPmax-Z}
In this subsection we prove the case when $A \subset \Z$ by induction on $k$.\footnote{We are grateful to Anton Bankevich for suggesting the proof in this subsection.} The statement is trivial when $k=2$ since $T_k(A) = n^2$ always. Now let $k \geq 3$, and assume that the statement has been proved for smaller values of $k$. It is convenient to count the number of nontrivial increasing $k$-APs in $A$:
\[ T_k'(A) = \#\{(a,b): b>0, a,a+b,\dots,a+(k-1)b \in A\}. \]
Clearly $T_k(A) = 2T_k'(A) + n$. Thus it suffices to show that $T_k'(A) \leq T_k'([n])$. Arrange the elements in $A$ in increasing order:
\[ A = \{a_0 < a_1 < \dots < a_{n-1}\}. \]
Choose $m \in \Z$ such that
\[ \frac{k-2}{k-1}(n-1) \leq m < \frac{k-2}{k-1} n + 1. \]
There are two types of $k$-APs counted in $T_k'(A)$: those whose second largest element is at least $a_m$, and those whose second largest element is smaller than $a_m$. If the second largest element of a $k$-AP in $A$ is $a_i$ for some $m \leq i < n$, then there are at most $n-1-i$ possibilities for its largest element. Thus the number of $k$-APs in $A$ of the first type is at most
\[ \sum_{m \leq i < n} (n-1-i). \]
For $k$-APs in $A$ of the second type, their first $k-1$ terms form a $(k-1)$-AP in $\{a_0,\dots,a_{m-1}\}$. Thus the number of $k$-APs in $A$ of the second type is at most
\[ T_{k-1}'(\{a_0,\dots,a_{m-1}\}) \leq T_{k-1}'(\{0,1,\dots,m-1\}) \]
by induction hypothesis. It follows that
\[ T_k'(A) \leq  \sum_{m \leq i < n} (n-1-i) + T_{k-1}'(\{0,1,\dots,m-1\}). \]
To conclude the proof, we claim that the first term on the right hand side above is equal to the number of $k$-APs in $\{0,1,\dots,n-1\}$ of the first type:
\begin{equation}\label{eq:maxAP-Z-typeI}
\sum_{m \leq i < n} (n-1-i) = \#\{(a,b): b>0, a \geq 0, a+(k-1)b < n, a+(k-2)b \geq m\},
\end{equation}
and the second term is equal to the number of $k$-APs in $\{0,1,\dots,n-1\}$ of the second type:
\begin{equation}\label{eq:maxAP-Z-typeII}
T_{k-1}'(\{0,1,\dots,m-1\}) = \#\{(a,b): b>0, a \geq 0, a+(k-1)b < n, a+(k-2)b < m\}.
\end{equation}
To prove~\eqref{eq:maxAP-Z-typeI}, it suffices to show that for any $m \leq i < n$ we have
\[ n-1-i = \#\{(a,b): b>0, a \geq 0, a+(k-1)b < n, a+(k-2)b =i \}. \]
This follows from the fact that any choice of the value of $j = a+(k-1)b$ from $\{i+1,\dots,n-1\}$ uniquely determines an admissible $(a,b)$ since
\[ i \geq (k-2)(j-i) \]
by our choice of $m$. To prove~\eqref{eq:maxAP-Z-typeII}, note that any $(a,b)$ with $b > 0$, $a \geq 0$, and $a+(k-2)b < m$ automatically satisfies $a+(k-1)b < n$ since
\[ a+(k-1)b = \frac{k-1}{k-2}\left[a + (k-2)b\right] - \frac{a}{k-2} \leq \frac{k-1}{k-2}(m-1) < n \]
by our choice of $m$.

\subsection{Proof of Theorem~\ref{thm:maxAP} when $A \subset \Z/N\Z$: preparations}\label{sec:kAPmax-prep}

In this subsection, we collect a few lemmas that will allow us to reduce the problem in $\Z/N\Z$ to the simpler one in $\Z$. From now on we fix some large prime $N$ and work in $\Z/N\Z$. To begin with, we show that $k$-AP counts are controlled by additive energy. For two subsets $A, B \subset \Z/N\Z$, the additive energy $E(A,B)$ is defined by
\[ E(A,B) = \#\{(a,a',b,b') \in A \times A \times B \times B: a+b = a'+b'\}, \]
and note the trivial bound
\[ E(A, B) \leq \min(|A|^2|B|, |A||B|^2) \leq |A|^{3/2}|B|^{3/2}. \]
We will also consider additive energy of dilates $\ell \cdot A$ for a positive integer $\ell$, defined by
\[ \ell \cdot A = \{ \ell a: a \in A\}. \]
For subsets $A_1, \dots, A_k \subset \Z/N\Z$, define the asymmetric $k$-AP count by
\[ T_k(A_1, \dots, A_k) = \#\{(a,b): a \in A_1, a+b \in A_2, \dots, a+(k-1)b \in A_k\}. \]
Clearly if $A_1 = \dots = A_k = A$ then $T_k(A, \dots, A) = T_k(A)$.

\begin{lemma}\label{lem:kAP-energy}
Let $A_1, \dots, A_k \subset \Z/N\Z$ be subsets, and let $n = \max(|A_1|, \dots, |A_k|)$. If
\[ \max_{\ell, \ell' \in \{1,2\}} E(\ell \cdot A_i, \ell' \cdot A_{i+1}) \leq \varepsilon n^3 \]
for some $\varepsilon \in (0,1)$ and $1 \leq i < k$, then $T_k(A_1, \dots, A_k) \leq \varepsilon^{1/6}n^2$.
\end{lemma}

\begin{proof}
Note that $T_k(A_1, \dots, A_k)$ is trivially bounded by either $T_3(A_{i-1}, A_i, A_{i+1})$ or $T_3(A_i, A_{i+1}, A_{i+2})$. The conclusion then follows immediately from~\cite[Lemma 4.2]{GS08}.
\end{proof}

The following lemma shows that Theorem~\ref{thm:maxAP} holds if $A$ has small doubling.

\begin{lemma}\label{lem:kAPmax-small-doubling}
Let $K \geq 1$. The following statement holds if $c$ is sufficiently small in terms of $K$. Let $A \subset \Z/N\Z$ be a subset with $|A| = n \leq cN$ and $|A+A| \leq K|A|$. Then $T_k(A) \leq T_k([n])$.
\end{lemma}

\begin{proof}
By~\cite[Theorem 4.1]{GS08}, there exisits a dilate $d \cdot A$ of $A$ (for some $d \in (\Z/N\Z)^*$) that is contained in an interval of length at most $N/2k$. After dilation and translation we may assume that $A \subset [1,N/2k]$. Now that any $k$-AP in $A$ (as a subset of $\Z/N\Z$) is also a $k$-AP in $\Z$, we may apply the $A \subset \Z$ case of Theorem~\ref{thm:maxAP} to conclude that $T_k(A) \leq T_k([n])$, as desired.
\end{proof}

\subsection{Proof of Theorem~\ref{thm:maxAP} when $A \subset \Z/N\Z$}\label{sec:kAPmax-Z/N}

We continue to work in $\Z/N\Z$. In view of Lemma~\ref{lem:kAPmax-small-doubling}, we may assume that $n$ is sufficiently large in terms of $k$. We first establish a rough structure theorem for sets with close to maximal number of $k$-APs.

\begin{lemma}\label{lem:kap-rough-structure}
Let $c > 0$ be sufficiently small. Let $A \subset \Z/N\Z$ be a subset with $|A| = n \leq cN$, and suppose that $n$ is sufficiently large in terms of $k$. If
\[ T_k(A) \geq \frac{n^2}{k-1} \left(1 - \frac{1}{100k^2}\right), \]
then there exists a dilate $d \cdot A$ of $A$ for some $d \in (\Z/N\Z)^*$, such that all but at most $n/10k^2$ elements in $d \cdot A$ lies in an interval of length at most $N/100k$.
\end{lemma}

\begin{proof}
Choose a sufficiently small $\varepsilon > 0$, and then choose some $\varepsilon' > 0$ sufficiently small in terms of $\varepsilon$. Apply~\cite[Proposition 3.2]{GS08} to obtain a structural decomposition $A = A_1 \cup \dots \cup A_m \cup A_0$ into disjoint subsets satisfying the following properties:
\begin{enumerate}
\item $|A_i| \gtrsim_{\varepsilon} |A|$ for each $1 \leq i \leq m$, so that $m\lesssim_{\varepsilon} 1$;
\item $|A_i+A_i| \lesssim_{\varepsilon,\varepsilon'} |A_i|$ for each $1 \leq i \leq m$;
\item $E(\ell \cdot A_i,\ell' \cdot A_j) \leq \varepsilon' |A_i|^{3/2} |A_j|^{3/2}$ whenever $1 \leq i < j \leq m$ and $\ell, \ell' \in \{1,2\}$;
\item $E(\ell \cdot A_0, \ell' \cdot A) \leq \varepsilon |A|^3$ whenever $\ell, \ell' \in \{1,2\}$.
\end{enumerate}
To estimate $T_k(A)$, we write it as the sum of $m^k$ terms of the form $T_k(A_{i_1}, \dots, A_{i_k})$ for some $1 \leq i_1, \dots, i_k \leq m$, and the $k$ terms $T_k(A_0,A,\dots,A), T_k(A,A_0,\dots,A), \dots, T_k(A,A,\dots,A_0)$. To estimate $T_k(A_{i_1}, \dots, A_{i_k})$ when $i_1,\dots,i_k$ are not all the same, we use Lemma~\ref{lem:kAP-energy} and property (3) to obtain
\[ T_k(A_{i_1}, \dots, A_{i_k}) \leq \varepsilon'^{1/6} n^2. \]
Thus the total contributions from these terms are bounded by $\varepsilon'^{1/6} m^{k} n^2$.  Moreover, by Lemma~\ref{lem:kAP-energy} and property (4) we also have
\[ T_k(A_0,A,\dots,A) \leq \varepsilon^{1/6} n^2, \]
and similarly for the other $k-1$ terms involving $A_0$. Thus we have shown that
\[ T_k(A) \leq \sum_{i=1}^m T_k(A_i) + \left( \varepsilon'^{1/6} m^{k} + \varepsilon^{1/6} k \right) n^2 \leq\sum_{i=1}^k T_k(A_i) + \frac{n^2}{100k^3}, \]
if $\varepsilon,\varepsilon'$ are small enough (recall that $m$ depends only on $\varepsilon$). If $c$ is sufficiently small in terms of $\varepsilon,\varepsilon'$, then Lemma~\ref{lem:kAPmax-small-doubling} can be applied to each $A_i$ in view of property (2) to get
\[ T_k(A) \leq \frac{1}{k-1} \sum_{i=1}^m |A_i|^2 + \frac{1}{4}(k-1)m + \frac{n^2}{100k^3} \leq \frac{n}{k-1} \max_{1\leq i \leq m} |A_i| + \frac{1}{4}(k-1)m + \frac{n^2}{100k^3}. \]
Combined with the lower bound for $T_k(A)$ in the hypothesis, this implies that $|A_i| \geq (1-1/10k^2)n$ for some $1 \leq i \leq k$. By~\cite[Theorem 4.1]{GS08}, after dilation we may assume that $A_i$ is contained in an interval of length at most $N/100k$.
\end{proof}

We are now ready to prove Theorem~\ref{thm:maxAP}. In view of Lemma~\ref{lem:kap-rough-structure}, after a suitable dilation and  translation we may assume that the set $A_0 = A \cap [-N/100k, N/100k]$ has size $|A_0| \geq (1-1/10k^2) |A|$. Let
\[ A_1 = A \cap ([-0.1N, 0.1N] \cup [0.4N, 0.6N]) . \]
Write $n_1 = |A_1|$ and $n_2 = |A\setminus A_1| = n-n_1$. Note that for any $k$-AP in $A$, if at least two of its first three terms lie in $A_0$, then it is entirely contained in $A_1$. Thus any $k$-AP in $A$ that is not entirely contained in $A_1$ must have at least one term outside $A_1$ and at least two out of the first three terms outside $A_0$, and the number of these $k$-APs is then bounded by $3kn_2|A\setminus A_0| \leq nn_2/2k$. It follows that
\[ T_k(A) \leq T_k(A_1) + \frac{nn_2}{2k}. \]
Note that $2 \cdot A_1 \subset [-0.2N, 0.2N]$. Thus from Lemma~\ref{lem:simple-rect} and the integer case of Theorem~\ref{thm:maxAP} proved in Section~\ref{sec:kAPmax-Z}, we obtain
\begin{equation}\label{eq:Tk(A1)}
T_k(A_1) = T_k(2 \cdot A_1) \leq T_k([n_1]).
\end{equation}
Using~\eqref{eq:Tk(n)} we arrive at
\[ T_k(A) \leq \frac{n_1^2}{k-1} + \frac{1}{4}(k-1) + \frac{nn_2}{2k}. \]
If $n_2 > 0$, then
\[ \frac{1}{4}(k-1) + \frac{nn_2}{2k} \leq \frac{nn_2}{k} \leq \frac{n^2-n_1^2}{k-1}, \]
provided that $n \geq k(k-1)/2$ (which we may assume). Thus in this case we have
\[ T_k(A) \leq \frac{n^2}{k-1} \leq T_k([n]) \]
as desired. If $n_2 = 0$, then $A = A_1$ and the desired conclusion follows from~\eqref{eq:Tk(A1)}.

\appendix

\section{Proofs of further bounds on Gaussian widths} \label{app:gw}

\subsection{Matching lower bound for 3-APs} \label{app:3ap-gw-lower}

Here we prove the matching lower bound
\[
\GW(T_3 /N) \gtrsim \sqrt{N \log N}
\]
in Theorem~\ref{lem:3ap-fourier-bound}. As in Section~\ref{sec:grad-3ap} we only work out the $\Omega = \ZZ/N\ZZ$ setting here. The $\Omega = [N] \subset \ZZ$ setting can be deduced similarly by embedding $[N]$ in a larger cyclic group to avoid APs from wrapping around zero.

First we show that if $h \colon \ZZ/N\ZZ \to \RR$ is a random function with $h(a) \sim \Normal(0,1)$ i.i.d.\ for $a \in \ZZ/N\ZZ$, then, with probability at least $1/2$, we can find some $f \colon \ZZ/N\ZZ \to [0,1]$ such that 
\begin{equation} \label{eq:3ap-lower}
\anglb{\tfrac{1}{N} \nabla T_3(f), h} \gtrsim \sqrt{N\log N}.
\end{equation}
Indeed, the real components of $\wh h(2s)$ for integers $0 \le s < N/4$ are independent Gaussians with variance $\Theta(1/N)$ (since the Fourier transform is orthogonal), so with probability at least $0.9$ there is some integer $0 < s < N/4$ such that $\Re \wh h(2s) \gtrsim \sqrt{(\log N) / N}$ and furthermore $\wh h(0) = O(1/\sqrt{N})$. Then, setting $f(a) = (1 + \cos(2\pi sa/N))/2$ so that $\wh f (0) = 1/2$, $f(\pm s) = 1/4$ and $\wh f(r) = 0$ for all $r \notin \{0,s,-s\}$, we obtain, by Lemma~\ref{lem:3ap-fourier-id}, 
\[
\frac{1}{N^2} T_3(f,h,f) = \wh f(0)^2 \wh h(0) + \wh f(s)^2 \wh h(-2s) + \wh f(-s)^2 \wh h(2s) = \frac14 \wh h(0) +  \frac{1}{8} \Re \wh h(2s) \gtrsim \sqrt{\frac{\log N}{N}},
\]
and 
\[
\frac{1}{N^2} T_3(h,f,f) = \frac{1}{N^2} T_3(f,f,h) = \wh f(0)^2\wh h(0) = O(N^{-1/2})
\]
Thus \eqref{eq:3ap-lower} holds since $\anglb{\nabla T_3(f),h} = T_3(f,f,h) + T_3(f,h,f) + T_3(h,f,f)$.

Finally, one can convert $f \colon \ZZ/N\ZZ \to [0,1]$ to a $\{0,1\}$-valued function by changing the value of $f$ at $a \in \ZZ/N\ZZ$ to $1$ with probability $f(a)$ and 0 with probability $1-f(a)$. A routine probabilistic argument shows that \eqref{eq:3ap-lower} holds for some $f \colon \ZZ/N\ZZ \to \{0,1\}$, thereby proving $\GW(T_3/N) \gtrsim\sqrt{N\log N}$.

\subsection{Improved upper bound for 4-AP}

Here we prove the final claim in Theorem~\ref{thm:gw-bound} that
\begin{equation} \label{eq:appen-gw4-bound}
\GW(T_4/N) \lesssim N^{3/4} (\log N)^{1/4}.
\end{equation}
As earlier we only discuss the $\Omega = \ZZ/N\ZZ$ setting here. We have
\begin{equation}
	\label{eq:4ap-decomp}
\anglb{\nabla T_4 (f), h} = T_4(h,f,f,f)+T_4(f,h,f,f)+T_4(f,f,h,f)+T_4(f,f,f,h).
\end{equation}

Define the following multiplicative analogue of the finite difference
\[
\Delta_s h (x) := f(x) \ol{h(x+s)}
\]
(since we will be working with real-valued functions, one can ignore the conjugation). The Fourier transform of $\Delta_s h$ controls $T_4$, as shown by the following lemma, which plays a similar role to Lemma~\ref{lem:3ap-fourier-bound} for $T_3$.

\begin{lemma} \label{lem:4ap-fourier-bound}
	For any $f \colon \ZZ/N\ZZ \to [-1,1]$ and any $h \colon \ZZ/N\ZZ \to \RR$,
	\[	
	\frac{1}{N^2} |T_4(h,f,f,f)| \lesssim \Bigl( \frac{1}{N} \sum_{s \in \ZZ/N\ZZ} \| \wh{\Delta_s h} \|_\infty\Bigr)^{1/2}.
	\]
	The same holds with the left-hand side replaced by $T(f,h,f,f)$, $T(f,f,h,f)$, or $T(f,f,f,h)$.
\end{lemma}

\begin{proof}
	We prove the inequality for $T(h,f,f,f)$ as the other cases are analogous. We have
	\begin{align*}
	|T_4(h,f,f,f)|^2
	&= \biggl\lvert \sum_a \biggl(\sum_b h(a-3b)f(a-2b)f(a-b)\biggr)f(a) \biggl\rvert^2
	\\
	&\le N\sum_a \biggl(\sum_b h(a-3b)f(a-2b)f(a-b)\biggr)^2 \qquad \text{\small [Cauchy-Schwarz]}
	\\
	&= N\sum_{a,b,b'} h(a-3b)h(a-3b')f(a-2b)f(a-2b')f(a-b)f(a-b')
	\\
	&= N\sum_{a,b,s} \Delta_{3s}h(a-3b) \Delta_{2s}f(a-2b) \Delta_{s} f(a-b) \qquad \text{\small [by setting $s = b - b'$]}
	\\
	&= N\sum_s T_3(\Delta_{3s} h, \Delta_{2s} f, \Delta_s f)
	\\
	&\lesssim N^3 \sum_s \|\wh{\Delta_{3s} h}\|_\infty \qquad \text{\small [by Lemma~\ref{lem:3ap-fourier-bound} on $T_3$]}
	\\
	& \lesssim N^3  \sum_s \|\wh{\Delta_s h}\|_\infty.
	\end{align*}
	The lemma follows by rearranging.
\end{proof}

Let $h \colon \ZZ/N\ZZ \to \RR$ be a random function with independent standard Gaussian values. Noting that 
\[
\wh{\Delta_s h}(r) = \frac{1}{N}\sum_{a \in \ZZ/N\ZZ} h(a)h(a+s)\omega^{-ar}
\]
is a quadratic form of Gaussians, and using tail bounds for such random variables \cite{HW71} (also \cite[Example 2.12]{BLM13}) we find that for every $s \ne 0$,
\begin{equation} \label{eq:fourier-diff-tail}
	\|\wh{\Delta_{s} h}\|_\infty = O\Bigl(\sqrt{\frac{\log N}{N}}\Bigr)  \quad \text{with probability } 1 - O(N^{-10}).
\end{equation}
We always have $\| \wh{\Delta_s h} \|_\infty \le 1$ for all $s$. Thus\begin{align*}
\GW(T_4/N) 	
&=
\EE_h \sup_{f \colon \ZZ/N\ZZ \to \{0,1\}} \anglb{\nabla T_4(f)/N, h}
\\
&\lesssim \EE_h \sqrt{N} \Bigl( \sum_{s \in \ZZ/N\ZZ} \| \wh{\Delta_s h}\|_\infty \Bigr)^{1/2} && \text{\small [by \eqref{eq:4ap-decomp} and Lem.~\ref{lem:4ap-fourier-bound}]}
\\
&\lesssim N^{3/4} (\log N)^{1/4}. && \text{\small [by \eqref{eq:fourier-diff-tail}]}
\end{align*}

\bigskip

\noindent\textbf{Acknowledgments.} We thank Ben Green and Freddie Manners for helpful discussions.  We also thank the anonymous referee for helpful comments that improved the exposition of the paper.

\bibliographystyle{abbrv}
\bibliography{ldpap_ref}

\end{document}